\documentclass[envcountsame,envcountchap]{article}
\usepackage{amssymb,amsmath}
\usepackage{amsfonts}
\usepackage{graphicx}
\usepackage[matrix,arrow,ps]{xy}

\usepackage{graphicx}
\usepackage{subfigure}
\usepackage[usenames]{color}
\usepackage{amssymb,amsmath,amsthm}
\usepackage{amsfonts}
\usepackage{graphicx}
\usepackage{type1cm}
\usepackage{eso-pic}
\usepackage{color}
\usepackage{makeidx}
\usepackage[section]{placeins}
\usepackage{float}

\newtheorem{theorem}{Theorem}[section]
\newtheorem{lemma}[theorem]{Lemma}
\newtheorem{cor}[theorem]{Corollary}

\newcommand{\blue}{\color{black}}

\begin{document}

\newcommand{\bz}{{\bf z}}
\newcommand{\vv}{{\bf v}}
\newcommand{\ww}{{\bf w}}
\newcommand{\yy}{{\bf y}}
\newcommand{\xx}{{\bf x}}
\newcommand{\nn}{{\bf n}}
\newcommand{\uu}{{\bf u}}
\newcommand{\mm}{{\bf m}}
\newcommand{\qq}{{\bf q}}
\newcommand{\OO}{\mathbb{O}}
\newcommand{\II}{\mathbb{I}}
\newcommand{\IR}{\mathbb{R}}
\newcommand{\IC}{\mathbb{C}}
\newcommand{\IB}{\mathbb{B}}
\newcommand{\IZ}{\mathbb{Z}}
\newcommand{\half}{\frac{1}{2}}
\newcommand{\halff}{1/2}
\newcommand{\bea}{\begin{eqnarray*}}
\newcommand{\eea}{\end{eqnarray*}}
\newcommand{\beaq}{\begin{eqnarray}}
\newcommand{\eeaq}{\end{eqnarray}}
\newcommand{\bfmu}{\mbox{\boldmath $\mu$ \unboldmath} \hskip -0.05 true in}
\newcommand{\bfnu}{\mbox{\boldmath $\nu$ \unboldmath} \hskip -0.05 true in}
\newcommand{\bfxi}{\mbox{\boldmath $\xi$ \unboldmath} \hskip -0.05 true in}
\newcommand{\bfphi}{\mbox{\boldmath $\phi$ \unboldmath} \hskip -0.05 true in}
\newcommand{\beq}{\begin{equation}}
\newcommand{\eeq}{\end{equation}}
\newcommand{\bfomega}{\mbox{\boldmath $\omega$ \unboldmath} \hskip -0.05 true in}
\newcommand{\bfeta}{\mbox{\boldmath $\eta$ \unboldmath} \hskip -0.05 true in}
\newcommand{\bfepsilon}{\mbox{\boldmath $\epsilon$ \unboldmath} \hskip -0.05 true in}
\newcommand{\bftheta}{\mbox{\boldmath $\theta$ \unboldmath} \hskip -0.05 true in}
\newcommand{\om}{\omega}
\newcommand{\Om}{\Omega}
\newcommand{\bom}{\bfomega}
\newcommand{\III}{\rm III}
\def\tab{ {\hskip 0.15 true in} }
\def\vtab{ {\vskip 0.1 true in} }
 \def\htab{ {\hskip 0.1 true in} }
 \def\ntab{ {\hskip -0.1 true in} }
 \def\vtabb{ {\vskip 0.0 true in} }
%-------------------------------------

\newtheorem{thm}[theorem]{Theorem}

\begin{center}
{\bf \LARGE Bootstrapping Globally Optimal Variational Calculus Solutions}
\end{center}

\begin{center}
Gregory S. Chirikjian \\
National University of Singapore \\
November 12, 2022
\end{center}

%=================================================================
%\Title{Bootstrapping Globally Optimal Variational Calculus Solutions}
%\TitleCitation{Bootstrapping Globally Optimal Variational Calculus Solutions}
%\newcommand{\orcidauthorA}{0000-0003-0542-9028}
%\Author{Gregory S. Chirikjian$^{1}$*}
%\AuthorNames{Gregory Chirikjian}
%\AuthorCitation{Chirikjian, G.S.}
%\address{$^{1}$ \quad National University of Singapore; mpegre@nus.edu.sg}
%% Current address and/or shared authorship
%%\firstnote{Current address: Department of Mechanical Engineering, NUS, Singapore}
\abstract{
Whereas in a coordinate-dependent setting the Euler-Lagrange equations establish necessary conditions for
solving variational problems in which both the integrands of functionals and the resulting paths
are assumed to be sufficiently smooth, uniqueness and global optimality are generally hard to
prove in the absence of convexity conditions, and often times they may not even exist.
This is also true for variational problems on Lie groups, with the Euler-Poincar\'{e} equation establishing
necessary conditions. The difficulties compound when integrands and/or the optimal paths are not sufficiently regular,
since in this case the classical necessary conditions no longer apply. This article therefore reviews several
nonstandard cases where unique globally optimal solutions can be guaranteed, and establishes a
``bootstrapping'' process to build new globally optimal variational solutions on larger spaces from existing
ones on smaller spaces. Surprisingly, it is possible to prove global optimality in some nonconvex cases
where even the regularity conditions required for classical necessary conditions do not hold.
This general theory is then applied to several topics such as optimal framing of curves in three-dimensional
Euclidean space, optimal motion interpolation, and optimal reparametrization of video sequences to compare salient actions.
}

% Keywords
%\keyword{Euler-Lagrange, Variational Calculus, Global Optimization}

%\begin{document}
%%%%%%%%%%%%%%%%%%%%%%%%%%%%%%%%%%%%%%%%%%

\section{Introduction}

In classical Variational Calculus, the Euler-Lagrange (E-L) Equations provide necessary conditions for local
optimality for functionals that have integrands that are sufficiently regular, under the assumption that the extremal
solutions are also sufficiently regular. Stronger necessary conditions for local optimality such as those due to Jacobi
and Legendre are also well known
\cite{vabrechtken,vaewing}. But in many cases testable sufficient conditions
for global optimality cannot be established unless the integrand is convex.
Otherwise, global optimality may not exist, may not be unique, or may not be provable.
A simple example of this is geodesics on the sphere. The Euler-Lagrange equations yield great arcs as the solution,
but for most pairs of points on the sphere there are two great arcs connecting them, one of which is globally minimal,
and one of which takes the long way around. When the points are antipodal, all great arcs connecting them are globally
optimal, hence there is a severe lack of uniqueness. In negatively curved (hyperbolic) spaces it is well known that
geodesics generated by solving the Euler-Lagrange equations are globally optimal and unique \cite{cheeger,hyperbolic}.
But this is a relatively rare situation, and uses specific geometric arguments. In contrast, the formulation presented
here will rely minimally on geometry, yet will have implications for geometric problems.

A broad class of variational problems for which minimality of the solution can be guaranteed is when the functional is
convex in both the generalized coordinates and their rates \cite{vagruver,vaecon}. But this is a tall order, since
even in mechanics where the Lagrangian is the integrand of the action functional, joint convexity of the Lagrangian in
coordinates and their rates is rare. Hence, it is well known that the ``Principle of Least Action'' is a misnomer.
It is really the ``Principal of Extremal Action'' since the Euler-Lagrange equations are generated and solved but little
attention is given as to whether the resulting solution is globally minimal.

\subsection{The Euler-Lagrange Equations}
\label{e-l-subsec}

Perhaps the most trivial variational integral for which the E-L equations provide unique globally optimal solutions  is that the arclength functional for curves in the Euclidean plane. The general single-degree-of-freedom variational problem is that of finding a function $y=y(x)$ such that the functional
\begin{equation}
J[y] \,=\, \int_{x_1}^{x_2} f(y(x),y'(x),x) \,dx
\label{ELfunc}
\end{equation}
is extremized
(where $y' = dy/dx$) for a given integrand $f(\cdot)$ subject to boundary conditions
$y(x_1) = y_1$ and $y(x_2) = y_2$. For example, when
$$ f(y(x),y'(x),x) = \sqrt{1+ [y'(x)]^2} $$
the problem becomes that of finding the curve $y=y(x)$ that extremizes arclength.

The E-L equation corresponding to (\ref{ELfunc}) is
\begin{equation}
\frac{\partial f}{\partial y} -
\frac{d}{dx} \left(\frac{\partial f}{\partial y'} \right) = 0.
\label{eulerlag}
\end{equation}
In the case of the arclength functional, this becomes $y'' =0$,
which when integrated twice gives $y(x) = ax+b$, the equation of a line. Even without invoking Euclid's elements, within the class of problems with solutions
of the form $y = y(x)$, the line is optimal because any perturbation added to it that observes the boundary conditions can
only increase $J$. The details as to why this is so will be addressed shortly, and have far reaching implications for
proving global optimality in other less trivial problems.

One might argue that the above is limiting because it does not allow for the possibility of curves in the plane connecting the points $(x_1,y_1)$ and $(x_2,y_2)$ that are not functions of the form $y=y(x)$. To consider the
more general case, we would need to seek parameterized curves of the form $x=y_1(t)$ and $y=y_2(t)$, and solve a
two-degree-of-freedom variational problem. The theory for such problems is also classical. Given
$$ J[{\bf y}] \,=\, \int_{t_1}^{t_2} f(y_1,...,y_N,\dot{y}_{1},...,\dot{y}_{N},t) \,dt $$
where $\dot{y}_{i} = dy_i/dt$, the set of E-L equations for $i=1,...,N$ are
\begin{equation}
\frac{\partial f}{\partial y_i} -
\frac{d}{dt} \left(\frac{\partial f}{\partial \dot{y}_{i}} \right) = 0.
\label{eulerlag1}
\end{equation}
This assumes that $f: \mathbb{R}^N \times \mathbb{R}^N  \times [t_1,t_2] \,\longrightarrow\, \mathbb{R}$ is sufficiently
smooth for all derivatives to exist \underline{and} the existence of extremal trajectories that are smooth enough for their
evaluation in the E-L equations to make sense. Two special classes of problems will be considered later in which: 1) $f$ is
not even differentiable, and hence the E-L equations do not apply, and
2) the Lavrentiev phenomenon occurs in which $f$ is regular but the resulting global minimizer is not,
and therefore does not satisfy the E-L equations. Yet global optimality can be guaranteed in both cases.

Returning to the curve in the Euclidean plane connecting two points, the arclength functional can be written as
$$ f(y_1,y_2,\dot{y}_{1},\dot{y}_{2},t) \,=\,
\sqrt{[\dot{y}_{1}]^2 + [\dot{y}_{2}]^2} \,, $$
and the resulting E-L equations are
$$ \ddot{y}_i = 0 \,\,\,{\rm for}\,\,\, i=1,2, $$
thus resulting in the parameterized line $y_i^*(t) = a_i t + b_i$, or
${\bf y}^*(t) = {\bf a} t + {\bf b}$ where ${\bf a}$ and ${\bf b}$ are constant vectors that provide the freedom to match the
boundary conditions. The strength of this argument relative to the previous one, again without invoking Euclid, is that it
considers the possibility of curves in the plane that connect two points and are not necessarily plotted as a
function $y=y(x)$.

More generally, given two points in a region of a manifold parameterized by coordinates $\qq = [q_1,...,q_N]^T$,
and metric tensor $M(\qq)$, the E-L equations (with $q_i$ taking the place of $y_i$) and the functional
\begin{equation}
f(\qq,\dot{\qq},t) = \sqrt{\dot{\qq}^T \,M(\qq)\, \dot{\qq}}
\label{kinen}
\end{equation}
give necessary conditions for a geodesic.
The square root is somewhat inconvenient when performing computations, and justification for eliminating it will be given shortly.

A metric tensor can arise in purely geometric contexts, e.g., for surfaces such as Minkowski sums of
elliposoids \cite{shiffman}, or it can arise from the kinetic energy of a system of $P$ particles with positions
${\bf x}_i \in \IR^3$, masses $m_i$, and generalized coordinates $\qq \in \IR^N$ (with $N \leq 3P$) in which case it is well known from classical mechanics that the entries of the mass metric tensor become
$$ m_{ij}(\qq) \,=\, \sum_{k=1}^{P} m_k \frac{\partial {\bf x}_k}{\partial q_i} \cdot
\frac{\partial {\bf x}_k}{\partial q_j} \,. $$
Hence the following well-known results from mechanics and a simple theorem will be helpful.

In Mechanics, the Lagrangian is defined to be the difference of kinetic and potential energy\footnote{Explicit dependence of the functional on $t$ is relatively rare in the
mechanics of particles and rigid bodies and in geometry, and henceforth is dropped from the discussion at this point. It will
be reintroduced in the discussion of the Lavrentiev phenomenon that follows later.}
$$ L(\qq,\dot{\qq}) \,=\,T(\qq,\dot{\qq}) \,-\, V(\qq) $$
and the total energy is
$$ E(\qq,\dot{\qq}) \,=\,T(\qq,\dot{\qq}) \,+\, V(\qq) \,. $$
The kinetic energy for a mechanical system always has the form
\begin{equation}
T(\qq,\dot{\qq}) \,=\, \half \dot{\qq}^T \,M(\qq)\, \dot{\qq} \,.
\label{kedef}
\end{equation}
The mass matrix of a mechanical system, $M(\qq)$, satisfies the required properties of a metric tensor.

Lagrange's equations of motion are the Euler-Lagrange equations (\ref{eulerlag1}) with $L$ in place of $f$, and $q_i$ in place of $y_i$, and they describe the motion of conservative systems. That is, if $\qq^*(t)$ denotes the solution to Lagrange's equations, it can be shown that energy is conserved:
\begin{equation}
E(\qq^*,\dot{\qq}^*) \,=\, const \,.
\label{energy-const}
\end{equation}
This classical fact will be used in the theorem presented below.

\begin{lemma}\label{phi-var}
Suppose $T$ is known a priori to be limited to the range $[x_1,x_2]$ where $0 < x_1 < x_2 < \infty$
and let $\Phi: [x_1,x_2] \rightarrow \IR_{>0}$ be a twice differentiable monotonically increasing function. Then the solution to the variational problems with a functional of the form
$f(\qq,\dot{\qq}) = \Phi(T(\qq,\dot{\qq}))$ will be the same as the solution to Lagrange's equations in the absence
of a potential energy term.\footnote{This result might be known, but the author could not find a clear statement
of it other than in the case when $\Phi(x) = c\sqrt{x}$ which is mentioned in several works
including \cite{brockett,DoCarmo,Dubrovin}.}
\end{lemma}

\begin{proof}
From the chain rule
$$ \frac{\partial f}{\partial q_i} = \Phi'(T) \frac{\partial T}{\partial q_i}
\,\,\,{\rm and} \,\,\, \frac{\partial f}{\partial \dot{q}_i} = \Phi'(T) \frac{\partial T}{\partial \dot{q}_i} \,, $$
and
$$  \frac{d}{dt} \left(\frac{\partial f}{\partial \dot{q}_i}\right) = \Phi''(T) \dot{T} \frac{\partial T}{\partial \dot{q}_i} +
\Phi'(T) \frac{d}{dt}\left(\frac{\partial T}{\partial \dot{q}_i}\right) \,. $$
But for a mechanical system with $V=0$, $T = const$ from (\ref{energy-const}), and so $\dot{T}=0$. Therefore
$$ \Phi'(T) \left[\frac{d}{dt} \left(\frac{\partial T}{\partial \dot{q}_i}\right) - \frac{\partial T}{\partial q_i} \right] \,=\,
\frac{d}{dt} \left(\frac{\partial f}{\partial \dot{q}_i}\right) - \frac{\partial f}{\partial q_i} = 0 \,. $$
Since $\Phi$ is monotonically increasing on the interval over which it is defined and continuous (since it is twice differentiable), it is also invertible. Moreover, $\Phi'(T)  > 0$. Therefore the variational problems for $f$ and $T$ are interchangeable.
\end{proof}
\noindent
In particular, restricting to when $T_0 > 0$ and letting $\Phi(x) = \sqrt{2x}$, we can solve the geodesic problem without the pesky square root, as if it were a mechanics problem.

The usefulness of this result is now illustrated with the line in the plane
where it is easier to prove the global optimality
for the integrand $T$ than the integrand $f = \Phi(T)$ in the original arclength functional. Starting with the variational solution ${\bf y}^*(t) = {\bf a} t + {\bf b}$ and letting ${\bf y}(t) = {\bf y}^*(t) + \bfepsilon(t)$, we see that the cost becomes
\beq
\int_{t_1}^{t_2} \|\dot{\bf y}\|^2 dt = \int_{t_1}^{t_2} \left\{
\|{\bf a}\|^2 + 2 {\bf a} \cdot \dot{\bfepsilon} + \|\dot{\bfepsilon}\|^2
\right\} \, dt
\,=\,
(t_2-t_1) \|{\bf a}\|^2 + 2 {\bf a} \cdot \int_{t_1}^{t_2} \dot{\bfepsilon} \,dt \,+\,
\int_{t_1}^{t_2} \|\dot{\bfepsilon}\|^2 \, dt
 \,.
 \label{epsiloneffect}
 \eeq
But since the perturbed solution must satisfy the same boundary conditions,
$$ \int_{t_1}^{t_2} \dot{\bfepsilon} \, dt = \bfepsilon(t_2) - \bfepsilon(t_1) = {\bf 0} - {\bf 0} = {\bf 0}. $$
From this it is clear that adding $\bfepsilon$ to ${\bf y}^*$ can only increase cost in (\ref{epsiloneffect}).
The same argument does not apply to the cost $\int_{t_1}^{t_2} \|\dot{\bf y}\| \, dt$ because the square root prevents the integral from acting directly on $\dot{\bfepsilon}$. That said, the Cauchy-Schwarz inequality provides
the bound
$$ \int_{t_1}^{t_2} \|\dot{\bf y}\| \, dt \,\leq\, \sqrt{t_2 - t_1}  \left(\int_{t_1}^{t_2} \|\dot{\bf y}\|^2 \, dt\right)^{\half} \,. $$
Substituting the minimal solution for the problem with integrand $\|\dot{\bf y}\|^2$
into the right hand side then provides an upper bound for the minimal solution to the variational problem with integrand $\|\dot{\bf y}\|$. And substituting the minimal solution for the problem with integrand $\|\dot{\bf y}\|^2$ into the cost with integrand $\|\dot{\bf y}\|$ provides an upper bound on the cost of the 
global minimizer for that problem. In Section 2, it will be shown that this second upper bound is in fact the global minimum. 

%More generally, if $\Phi(\cdot)$ is convex then Jensen's inequality can be used as
%$$ *** $$
In other words, changing the original problem by an appropriate choice of $\Phi$ can preserve the
E-L necessary conditions while making it easy to guarantee global optimality of the solution to the new problem.
Though it is not obvious using elementary arguments that the global optimality of the new problem
guarantees it for the original problem, such statements are made in differential geometric settings.
This issue will be revisited without geometry in Section \ref{suffelsec}. But first a small detour will be taken to discuss
variational problems for which the resulting optimal trajectories lack regularity (i.e., they are non-Lipschitz), and for which the E-L equations are
not necessary conditions.

\subsection{The Lavrentiev Phenomenon}
\label{irregular}

Whereas the classical Variational Calculus of Euler, Lagrange, Jacobi, etc.
assumes functionals with sufficiently smooth integrand and well-behaved resulting
extremal trajectories, over the past 100 years
variational problems with globally minimal solution have been investigated in the absence of classical regularity conditions on the integrand and/or the
globally minimal trajectories that result.

The {\it Lavrentiev phenomenon} has been of particular interest \cite{Lavrentiev,Mania,Lowen}. As an example, consider trajectories $x(t)$ that minimize the functional
\begin{equation}
J[x] = \int_{0}^{1} (x^3 - t)^2 \dot{x}^6 dt
\label{maniaex}
\end{equation}
subject to boundary conditions $x(0) = 0$ and $x(1) = 1$.
This is the Mani\'{a} example \cite{Mania}.
Though the integrand is smooth, it can be shown that the globally minimal solution is $x^*(t) = t^{1/3}$.
As this is not Lipschitzian, but is absolutely continuous, this is an example
where the infimum of the functional taken over the set of absolutely continuous trajectories is lower than that taken over Lipschitz trajectories. This is the
essence of the Lavrentiev phenomenon.

Functionals such as (\ref{maniaex}) and others that demonstrate the Lavrentiev phenomenon, such as those in \cite{Buttazzo} are remarkable in that their integrand is
regular, and yet the resulting trajectories are not \cite{Ball1}
Consequently, the E-L equations need not be necessary conditions \cite{Ball2}.
This remains an active topic of research with more such examples continuing to emerge
\cite{foss,Cellina,Ferriero}.

The following section examines an altogether different kind of variational problem with globally minimal solution. Then in Section \ref{bootstrapsec} all of these pieces from previous sections are assembled using a ``bootstrapping'' procedure in which globally minimal solutions on higher-dimensional spaces
are constructed.

\section{Nonstandard Case Studies in Global Optimality}
\label{suffelsec}

Though the Euler-Lagrange equations provide necessary conditions
for local optimality under the conditions that: 1) $f$ is sufficiently smooth; and 2) a sufficiently smooth solution exists,
they provide no guarantee that such a solution will be globally optimal. So-called second-order necessary conditions provide more confidence that the solution to
the E-L equations might be minimal, but also provide no guarantee. In the absence of specific a priori
knowledge of the properties of the functional, such as convexity, it is usually quite difficult to say anything other than that the solution generated by the E-L equations are extremals (rather than global minima or maxima), unless additional arguments can be employed as in the case of the line in the previous section.

That said, in certain special situations, the structure of the function
$f(\cdot)$ will guarantee global optimality of the solution generated by the
the Euler-Lagrange equations. Here special cases are presented in which global optimality of solutions of variational problems
are guaranteed when satisfying the Euler-Lagrange necessary conditions. The emphasis here is on this sense of the word `global' in contrast
to global descriptions of geometry or dynamics which can be found elsewhere \cite{cheeger,Lee}.

\subsection{Global optimality in 1D case} \label{1Dglobal}

Consider the set of monotonically increasing bijective functions that map the closed interval $[0,1]$ into itself and are twice differentiable on the open interval $(0,1)$ and satisfy the boundary conditions $y(0) = 0$ and $y(1) = 1$.
Such functions include but are not limited to $y(x) = x^p$ for any positive power $p \in \IR_{>0}$. The set of all such functions are closed under function inversion and function composition $(y_1 \circ y_2)(x) = y_1(y_2(x))$,
and have an identity element, $y_{id}(x) = x$. The result is an infinite dimensional group.
In this section a variational problem is formulated that selects the one element of this group
that globally minimizes a cost functional.

There exists a very general functional in the one-dimensional
case for which the global minimality of the solution to the Euler-Lagrange
equation is guaranteed. Specifically, let
$$ f(y,y',x) =  m(y) (y')^2 $$
where $m:[0,1]\,\longrightarrow\,\mathbb{R}_{>0}$ is differentiable and let
\begin{equation}
J[y] = \int_{0}^{1}  m(y) (y')^2 dx \,.
\label{Jy}
\end{equation}
Global minimization of this sort of functional has been examined informally in \cite{stochastic,harmonic}. A more rigorous treatment is given below.

\subsubsection{Globally Minimal Reparametrization}

\begin{lemma} \label{mainlemma}
Let $m_1$ and $m_2$ be real numbers with $0 < m_1 < m_2 < \infty$ and let $m:[0,1] \,\rightarrow\, [m_1,m_2]$ be differentiable. Let $m^{\half}(s)$ and $m^{-\half}(s)$ respectively be shorthand for $\sqrt{m(s)}$ and $1/\sqrt{m(s)}$. The solution to the resulting Euler-Lagrange equations for (\ref{Jy}) subject to the boundary conditions $y(0)=0$ and $y(1)=1$ satisfies the implicit equation
\begin{equation}
y^*(x) = \frac{\int_{0}^{x} m^{-\half}(y^*(s)) ds}{\int_{0}^{1} m^{-\half}(y^*(s)) ds}
\label{ystar}
\end{equation}
and $y^*(x)$ can be obtained explicitly by inverting the function $F:[0,1]\,\rightarrow\,\left[\,0,\sqrt{m_2}\,\right]$ defined by
$$ F(y) \,\doteq\, \int_{0}^{y} m^{\half}(s) ds $$
in the expression
\begin{equation}
F(y^*(x))  = F(1) \, x \,.
\label{ystar333}
\end{equation}
Moreover
\begin{equation}
\int_{0}^{1} m^{-\half}(y^*(s))\,ds \,=\, 1/F(1) \,.
\label{c1eqk42b4}
\end{equation}
\end{lemma}
\begin{proof}
Recognizing that $m'=(dm/dy)y'$, the E-L equation corresponding to (\ref{Jy}) will be
$$ m y'' + \half m' y' = 0 \,. $$
Multiplying through by $m^{-\half}(x) = 1/\sqrt{m(x)}$, the result then can be written as
$$ \left(m^{\half} y'\right)' = 0 \,. $$
The solution is then of the form
\begin{equation}
m^{\half}(y^*) \frac{dy^*}{dx} = c_1
\label{keyc1}
\end{equation}
where $c_1$ is a constant of integration,
and so
$$ y^*(x) = c_2 + c_1 \int_{0}^{x} m^{-\half}(y^*(s)) ds \,. $$
Since $y(0)=0$ and $y(1)=1$ this becomes (\ref{ystar}).

Alternatively,
$$ m^{\half}(y^*) {dy^*} = c_1 {dx} $$
can be integrated to give
$$ F(y^*) = c_1 x\,.$$
The condition $F(y^*(1)) = c_1$ together with the boundary condition $y^*(1)=1$ then gives (\ref{ystar333}).
That is, $c_1 = F(1) \in \left[\,\sqrt{m_1},\sqrt{m_2}\,\right]$. But since the same $c_1$ appears in the normalization of the implicit expression,
(\ref{c1eqk42b4}) results.
\end{proof}

The following theorem proves that the result of the above lemma is not only a solution to the
E-L equations, but that it is globally optimal. Moreover, this global optimality generalizes to cases where the E-L equations no longer apply due to a lack of sufficient regularity, i.e., the case when the function $m$ is not differentiable.

\begin{thm} \label{1Dglobthm}
Let all of the conditions of Lemma \ref{mainlemma} hold, except that the function $m:[0,1]\,\rightarrow\,[m_1,m_2]$ need only be continuous rather than differentiable. Then the implicit solution (\ref{ystar}) and explicit solution (\ref{ystar333}) globally minimize (\ref{Jy})
among all differentiable functions subject to the boundary conditions $y(0) = 0$ and $y(1) = 1$.
\end{thm}
\begin{proof}
(\ref{ystar}) and (\ref{ystar333}) both satisfy
(\ref{keyc1}) without envoking the E-L equations.
Substituting $y^*(x)$ into (\ref{Jy}) and observing (\ref{keyc1}) gives
$$ J[y^*] = \int_{0}^{1}  \left(m^{\half}(y^*) \frac{dy^*}{dx}\right)\left(m^{\half}(y^*) \frac{dy^*}{dx}\right) dx \,=\, c_1^2 $$
where $c_1 = F(1)$. The continuity and boundedness of $m$ ensures that $F$ is
bounded, differentiable, and monotonically increasing. This
together with the form of (\ref{ystar}) or (\ref{ystar333})
ensures differentiability of $y^*$.
Since $c_1^2 = \left(\int_{0}^{1} c_1 dx\right)^2$,
$$ J[y^*] = \left(\int_0^1 m^{\half}(y^*(x)) \frac{dy^*}{dx} dx\right)^2
= \left(\int_0^1 m^{\half}(y^*) {dy^*}\right)^2.$$
The above is simply a change of variables of integration and
has nothing to do with the details of (\ref{ystar}) other than
the fact that $y^*(0)=0$ and $y^*(1) =1$ and ${dy^*}/{dx} > 0$, which is true even though $m$ may not be monotonic, convex, or differentiable.
Consequently, the value of the integral in the above expression for
$J[y^*]$ is {\it independent} of the path $y^*(x)$.
That is,
$$ J[y^*] = \left(\int_0^1 m^{\half}(y) {dy}\right)^2 $$
can be computed without reference to $y^*(t)$
since the name of the variable of integration is irrelevant.
Using the Cauchy-Schwarz inequality
$$ \int_{0}^{1} a(x) b(x) dx \,\leq\, \left(\int_{0}^{1} |a(x)|^2 dx\right)^{\half}
\left(\int_{0}^{1} |b(x)|^2 dx\right)^{\half} $$
with $a(x) = m^{\half}(y(x)) {dy}/dx$ and $b(x) = 1$ then gives
\beq J[y^*] \leq J[y] \eeq
{\it for every possible} differentiable $y(x)$ that satisfies the boundary conditions.
\end{proof}

What is remarkable about this simple example is that none of the classical conditions stated in variational calculus regarding sufficient conditions for optimality apply. This is not a hyperbolic space,
the functional is not convex, etc. And in addition to all of this, the
regularity conditions on $m$ needed for the E-L equations to hold can be relaxed.  Moreover, as will be seen later, from
this simple example it is possible to bootstrap up
to {\blue cases in which $y(t)$ is replaced by a path defined in a multi-dimensional space and where
global optimality is still guaranteed.}

\subsubsection{Reparameterizing an Extremal Solution Only Makes Things Worse}

Given $T = \frac{1}{2} \dot{\qq}^T M(\qq) \dot{\qq}$ and $V=0$, with optimal solution $\qq^*(t)$ that conserves
kinetic energy $T(\qq^*,\dot{\qq}^*) = T_0$, it is clear that arclength along such a trajectory
is
$$ s(t) \,=\, \int_{0}^{t} \sqrt{\dot{\qq}^*(t')^T M(\qq^*(t')) \dot{\qq}^*(t')}\, dt' \,=\, \int_{0}^{t} \sqrt{2T_0} \, dt'
\,=\, \sqrt{2T_0} \, t\,. $$
If we were to attempt to reparameterize by substituting $\qq^*(t) \,\rightarrow \, \qq^*(\tau(t))$, with
$\tau$ being the same sort of monotonically increasing function on $[0,1]$ discussed earlier, the resulting cost would be
$$ \int_{0}^{1} \sqrt{\dot{\qq}^*(\tau(t))^T M(\qq^*(\tau(t))) \dot{\qq}^*(\tau(t))} \dot{\tau}^2(t) \,dt \,=\, \sqrt{2T_0} \int_{0}^{1} \dot{\tau}^2(t) \, dt \,. $$
The optimal $\tau(t)$ in this context is $\tau^*(t) = t$ in which case
$\int_{0}^{1} (\dot{\tau}^*)^2(t) \, dt = 1$. For any other $\tau(t) = t + \epsilon(t)$, the
result will be
$$ \int_{0}^{1} \dot{\tau}^2(t) \, dt \,=\, \int_{0}^{1} (1 + \dot{\epsilon})^2 \, dt \, =\,
1+ \int_{0}^{1} \dot{\epsilon}^2 \, dt \,, $$
which only can increase cost.

\subsection{Implications for Solving Auxiliary Variational Problems}

A result of combining Lemma \ref{phi-var} and Theorem \ref{1Dglobthm} is
\begin{cor}
Let $T$ be as in (\ref{kedef}). Then
the trajectory $\qq^*(t)$ that minimizes\footnote{$J_1[\qq]$ and $J_2[\qq]$ are not functions of
$\qq$, but rather are functionals, the value of which is completely specified by the path $\qq=\qq(t)$.}
$$ J_2[\qq] \doteq \int_{0}^{1} 2T(\qq,\dot{\qq}) \, dt $$
globally subject to the end constraints $\qq(0)$ and $\qq(1)$ will also minimize
$$ J_1[\qq] \doteq \int_{0}^{1} \sqrt{2T(\qq,\dot{\qq})} \, dt $$
globally subject to the same end constraints.
That is, for any twice differentiable trajectory $\qq(t)$ satisfying the boundary conditions
\begin{equation}
J_2[\qq^*] \leq  J_2[\qq] \,\,\,\Longrightarrow\,\,\, J_1[\qq^*] \leq  J_1[\qq] \,.
\end{equation}
Moreover
\beq
J_2[\qq^*] = (J_1[\qq^*])^2 \,.
\label{J2eqJ1abc}
\eeq
\end{cor}

\begin{proof}
If $\Phi(x) = \sqrt{x}$, then by Lemma \ref{phi-var} the E-L equations for both functionals are the same
and since the boundary conditions are the same, the resulting solutions will both be $\qq^*(t)$.

Let $\qq = \qq(t)$ be an arbitrary differentiable trajectory on $[0,1]$ that satisfies the end constraints, and let $\tau(t)$ be a reparameterization of time. Then $(\qq \circ \tau)(t) = \qq(\tau(t))$ will also be a feasible trajectory. Suppressing the temporal dependence on $\tau$, this will be referred to as $\qq(\tau)$. Let
$$ m(\tau) \,=\, \frac{d\qq}{d\tau}^T M(\qq) \,\frac{d\qq}{d\tau} \,. $$
Then $2T(\qq(\tau), \frac{d\qq}{d\tau} \dot{\tau}) = m(\tau) \dot{\tau}^2$.
If $\tau^*$ is the global minimizer of this one-dimensional functional, and if $\qq^*(t)$ is the global minimizer of
the functional $J_2$, then
\beq
J_2[\qq^*] \,\leq\, J_2[\qq(\tau^*)] \,\leq\, J_2[\qq] \,.
\label{posineq}
\eeq
Recall that the differentiable trajectory $\qq(t)$ is arbitrary up to the same boundary conditions, and $\qq(\tau^*(t))$ is the same path in configuration space parameterized differently by time from $\qq(t) = \qq(\tau(t)=t)$.

From Theorem \ref{1Dglobthm}, we know that $\qq(\tau^*(t))$ will give
\beq
J_2[\qq(\tau^*)] = (J_1[\qq(\tau^*)])^2 \,.
\label{J2eqJ1}
\eeq
This is true for any differentiable $\qq(\tau)$ satisfying the end constraints and evaluated at $\tau = \tau^*(t)$. The significance of (\ref{J2eqJ1}) is that
since both $J_1$ and $J_2$ are positive, then for any given $\qq(\tau)$ they must both be minimized by the same $\qq(\tau^*(t))$.

In the case when
$\qq = \qq^*$ (the postulated global minimal solution for $J_2$), reparameterizing the temporal variable results in $\tau^*(t) = t$
since by definition there is no way to reparameterize to improve the global minimum. Therefore,
$J_2[\qq^*(\tau^*)] = J_2[\qq^*]$.

On the other hand, the value of the integral
$$ J_1[\qq(\tau)] = \int_{0}^{1} \sqrt{m(\tau)} \, \dot{\tau} \, dt \,=\,
\int_{0}^{1} \sqrt{m(\tau)} \, d\tau $$
is independent of the choice of the differentiable monotonically increasing bijective function $\tau:[0,1] \rightarrow [0,1]$ regardless of which path $\qq(t)$ is taken. That is, for any valid $\tau(t)$,
$$ J_1[\qq(\tau)] = J_1[\qq] \,. $$
Therefore, $J_1[\qq^*(\tau)] = J_1[\qq^*]$ for every such $\tau$, including $\tau^*$ and
\beq
J_1[\qq^*] \,\leq\, J_1[\qq(\tau^*)] \,=\, J_1[\qq] \,.
\label{fromposineq}
\eeq
Moreover, from (\ref{J2eqJ1}), $J_1$ is globally minimized on the same global minimizer of $J_2$, namely $\qq^*(t)$ and the minimal values are related as (\ref{J2eqJ1abc}).
%$$ J_2(\qq^*) = (J_1(\qq^*)))^2 \,.$$
%Therefore, if $J_2(\qq(\tau^*))$
\end{proof}

\subsection{The Poincar\'{e} Half Space Model}

As an example where the E-L equations provide a globally minimal solution, consider the Poincar\'{e} Half Space Model, which is a Riemannian space of constant negative curvature.

Given an $N\times N$ metric tensor $M(\qq) = [m_{ij}(\qq)]$ and its inverse using the raised index notation $[M(\qq)]^{-1} = [m^{ij}(\qq)]$, the Christoffel symbols are computed as
\begin{equation}
\Gamma^{i}_{\,jk} \doteq \frac{1}{2} \sum_{l=1}^{N} m^{il} \left(
\frac{\partial m_{lj}}{\partial q_k} +
\frac{\partial m_{lk}}{\partial q_j} -
\frac{\partial m_{jk}}{\partial q_l} \right)\,,
\label{christ}
\end{equation}
and from these the
Riemannian curvature tensor is computed as
\begin{equation}
R^{i}_{\,jkl} \doteq -\frac{\partial \Gamma^{i}_{\,\,jk}}{\partial q_l}
+ \frac{\partial \Gamma^{i}_{\,\,jl}}{\partial q_k} +
\sum_{m=1}^{N} (- \Gamma^{m\,}_{\,\,jk} \Gamma^{i}_{\,\,ml} +
\Gamma^{m\,}_{\,\,jl} \Gamma^{i}_{\,\,mk}) \,.
\label{curve}
\end{equation}
The notation of raising and lowering indicies gives
$$ R_{hjkl} = \sum_{i=1}^{N} m_{hi} R^{i}_{\,jkl} \,. $$
A space is said to have constant curvature if \cite{eisenhart} %p83
\begin{equation}
R_{hjkl} \,=\, K_0 \left(m_{hk} m_{jl} - m_{hl} m_{jk}\right)\,.
\label{curveconstsjsvv}
\end{equation}

Hyperbolic space defined by constant negative curvature at every point is an example where it is known
that geodesics connecting any two points are unique and globally length minimizing. The Poincar\'{e} half space model of hyperbolic space is the open half space
$$ \mathbb{H}^+ \,\doteq\, \{\xx \in \IR^N \,,\, {\bf x} \cdot {\bf e}_N > 0\} $$
endowed with the metric tensor
$$ M(\xx) =  ({\bf x} \cdot {\bf e}_N)^{-2} \, \II\,. $$
For example, in the 2D case this consists of point $(x_1,x_2>0) \in \IR^2$ with
$$ \dot{\bf x}^T M(\xx) \dot{\bf x} = \frac{\dot{x}_1^2 + \dot{x}_2^2}{x_2^2} \,. $$
In this case it can be shown that (\ref{curveconstsjsvv}) holds with
\begin{equation}
K_0 \,=\, -1 \,.
\label{curveconstsjsvvfffe}
\end{equation}
The family of geodesics in this case is parameterized by the group $PSL(2,\IR) \,\doteq\, SL(2,\IR)/\{\II,-\II\}$
where
$$ SL(2,\IR) \,=\, \left\{\left.\left(\begin{array}{cc} a & b \\ c & d \end{array}\right) \,\right|\, ad - bc \,=\,1\right\}\,. $$

Let $i = \sqrt{-1}$ and let $a,b,c,d$ be real numbers satisfying $ad - bc = 1$.
Define the curve in the complex plane
$$ z^*(t) = \frac{b + ia e^t}{d + i c e^t} = x_1^*(t) + i x_2^*(t) \,. $$
That is, $x_1^*(t) = Re(z^*(t))$ and  $x_2^*(t) = Im(z^*(t))$. (Optimality properties of this curve will be demonstrated shortly, and $*$ denotes this, i.e., it is not Hermitian conjugation).
Explicitly,
\begin{equation}
x_1^*(t) = (bd + ac e^{2t})(d^2 + c^2 e^{2t})^{-1} \,\,\,{\rm and}\,\,\,
x_2^*(t) = e^{t} (d^2 + c^2 e^{2t})^{-1} \,.
\label{sknddwww}
\end{equation}
The signed curvature of these curves as viewed in the Euclidean plane, computed using the general formula
\begin{equation}
k(t) =
\frac{\dot{x}_1 \ddot{x}_2 - \dot{x}_2 \ddot{x}_1}{\left(\dot{x}_1^2 + \dot{x}_2^2\right)^{3/2}} \end{equation}
is $k^*(t) = -2cd$, which is constant, indicating a clockwise bending circular arc.
Since these are circles, they can be written as
\begin{equation}
(x_1^* - x_0)^2 + (x_2^*)^2 = r^2
\end{equation}
where $r = 1/|k^*| = 1/(2cd)$ and
$$ x_0 = \frac{1}{2}\left(\frac{b}{d} + \frac{a}{c}\right) \,. $$

It can be shown that on these circles
\begin{equation} \frac{(\dot{x}_1^*)^2 + (\dot{x}_2^*)^2}{(x_2^*)^2} = 1, \end{equation}
and consequently $t$ is arclength in the hyperbolic metric.

Moreover, the Euler-Lagrange equations for the functional
\begin{equation}
J = \int_{t_0}^{t_1} \frac{\dot{x}_1^2 + \dot{x}_2^2}{x_2^2} dt
\label{Jdefhyp}
\end{equation}
can be written as
\begin{eqnarray}
\frac{d}{dt}\left(\frac{\dot{x}_1}{x_2^2}\right) \,&=&\, 0 \label{firsthypergo} \\
\ddot{x}_2 \,&=&\, x_{2}^{-1} \left(\dot{x}_2^2 - \dot{x}_1^2\right) \,.
\end{eqnarray}
$x_i^*(t)$ given above satisfy these equations.
Consequently (\ref{sknddwww}) is an extremal of (\ref{Jdefhyp}).
In particular, (\ref{firsthypergo}) gives
$$ \frac{\dot{x}_1^*}{(x_2^*)^2} = 2cd \,. $$

From the fact that ${\dot{x}_1^*}/{(x_2^*)^2}$ is constant, it is clear that if $\epsilon_1(t)$
is a nonzero arbitrary differentiable function
that vanishes at $t=t_0$ and $t=t_1$, it must be that
$$ \int_{t_0}^{t_1} \frac{(\dot{x}_1^*)^2 + (\dot{x}_2^*)^2}{(x_2^*)^2} dt
< \int_{t_0}^{t_1} \frac{(\dot{x}_1^*+\dot{\epsilon}_1)^2 + (\dot{x}_2^*)^2}{(x_2^*)^2} dt $$
based on the same argument as for the line. That is, when expanding the term in parenthesis, the resulting
linear term in $\dot{\epsilon}$ integrates to zero.

Moreover, since these curves $\xx^*(t)$ are already parameterized by arclength, replacing them with
$\yy(t) = \xx^*(\tau(t))$ where $\tau(t) \neq t$ is monotonically increasing and  $\tau(t_0) = t_0$ and $\tau(t_1)  = t_1$ can only increase the value of $J$.
Consequently,
$$ \int_{t_0}^{t_1} \frac{(\dot{x}_1^*)^2 + (\dot{x}_2^*)^2}{(x_2^*)^2} dt
< \int_{t_0}^{t_1} \frac{\dot{y}_1^2 + \dot{y}_2^2}{y_2^2} dt \,. $$
Writing $\yy(t) = \xx^*(t) + \bfnu(t)$ (where there is one degree of freedom in $\bfnu(t)$ following from the
freedom of the scalar function $\tau(t)$) and combining the above two inequalities provides evidence for the following statement.

If $\epsilon_1(t)$ and $\epsilon_2(t)$ are two independent nonzero arbitrary differentiable functions
that vanish at $t=t_0$ and $t=t_1$, then
%Can the terms in
%$$ \frac{(\dot{x}_1+\dot{\epsilon}_1)^2 + (\dot{x}_2+\dot{\epsilon}_2)^2}{(x_2+\epsilon_2)^2} $$
%be reorganized to show that
$$ \int_{t_0}^{t_1} \frac{(\dot{x}_1^*)^2 + (\dot{x}_2^*)^2}{(x_2^*)^2} dt
< \int_{t_0}^{t_1} \frac{(\dot{x}_1^* + \dot{\epsilon}_1)^2 + (\dot{x}_2^* + \dot{\epsilon}_2)^2}{(x_2^* + \epsilon_2)^2} dt \,, $$
indicating that (\ref{sknddwww}) are globally minimal geodesics.

%Try to simplify
%
%$$ \frac{(\dot{x}_1+\dot{\epsilon}_1)^2 + (\dot{x}_2+\dot{\epsilon}_2)^2}{(x_2+\epsilon_2)^2}
%- \frac{\dot{x}_1^2 + \dot{x}_2^2}{x_2^2} $$
%
%*** The above parts with $\epsilon$ are probably too hard. Try the calculation below ***

{\blue
\subsection{Special Higher Dimensional Vectorial Cases}
}

In the previous subsection, the Poincar\'{e} half-plane model was used.
Another model of the hyperbolic plane is the Poincar\'{e} disk model in which
$$ f(\xx,\dot{\xx}) = \dot{\xx}^T M(\xx) \dot{\xx} \,\,\,{\rm where} \,\,\,
M(\xx) \,=\, \frac{4}{1-\|\xx\|^2} \II \,.$$
Extensions of the Poincar\'{e} half-plane and disk models to higher dimensions exist. For example,
the above expression for $M(\xx)$ holds in $n$ dimensions. The $n$-dimensional half-space model
has $M(\xx) = x_{n}^{-2} \II$. From geometric considerations it is known that geodesics in these spaces
are global minimizers of length.

All of the above are of the form
$$ f(\xx,\dot{\xx}) = m\left(\sum_{i=1}^{N} a_i x_i^2\right) \sum_{j=1}^{N} b_j \dot{x}_j^2 $$
where $b_j = 1$ and $a_i \in \{0,1\}$.
That is, the mass matrix consists of a scalar mass function, $m(\cdot)$, which depends on coordinates
in a particular way, which multiplies the identity matrix. Therefore, functions of the form above
are a reasonable canididate for exploring wider conditions on the parameters $\{a_i\}$ and $\{b_j\}$
such that global minimality of the variational problem is guaranteed.

Another very special {\blue vectorial} case is when there are coordinates $\{{q_i}'\}$ such that
$$ f(\qq',\dot{\qq}',t) = \sum_{i=1}^{N} \lambda_i({q_i}') ({\dot{q}_i}')^2 $$
where each $\lambda_i$ is a continuous function on the interval $[0,1]$ that takes finite positive values.
Then global optimization in each coordinate can be performed independently as
per Theorem 2.2.
Likewise, if a coordinate transformation can reduce a coupled {\blue vectorial} problem to this form,
then the global optimum can be found in each new coordinate, and the result can be transformed back.
Such cases are not the norm. That is, if the matrix
$$ M(\qq) \,=\, S^T \Lambda(S \qq) \,S $$
for some invertible constant matrix $S$, then the variational problem for
$ f(\qq,\dot{\qq},t) \doteq \dot{\qq}^T M(\qq) \,\dot{\qq} $
will have a globally minimal solution if it is defined on the region $S \cdot [0,1]^N$
because then in the coordinates $\qq' = S \qq$, everything will reduce to
$f(\qq',\dot{\qq}',t)$.

\subsection{A Suboptimal Ansatz}

For more complicated variational problems, employing numerical shooting methods to meet boundary conditions can be computationally
costly. Moreover, since globally optimal solutions might not even exist, it can sometimes be useful to obtain a fast suboptimal solution
to a variational problem rather than an exact extremal.

Motivated by the splitting in Theorem \ref{1Dglobthm}, we can define a trajectory $\qq^{\circ}(t)$ such that for symmetric positive definite $M(\qq) \in \IR^{n\times n}$ (as in kinetic energy or metric tensor) and $R(\qq) \in SO(n)$,
$$ R^T(\qq^{\circ}) \, M^{\half}(\qq^{\circ}) \, \dot{\qq}^{\circ} = {\bf c} $$
where ${\bf c}$ is a constant vector. Then the kinetic energy $T=\frac{1}{2} \dot{\qq}^T M(\qq) \dot{\qq}$ along this path will remain constant and
\beq
\qq^{\circ}(t)  = \qq(0) + \int_{0}^{t} M^{-\half}(\qq^{\circ}(\tau)) R(\qq^{\circ}(\tau),\tau) \, {\bf c} \, d\tau \,.
\label{mhalf}
\eeq
Here $M^{\half}(\qq)$ is the matrix square root of $M(\qq)$ such that $M^{\half}(\qq)M^{\half}(\qq) = M(\qq)$
and $ M^{-\half}(\qq) = [M^{\half}(\qq)]^{-1}$ and $R(\qq,t) \in SO(n)$ is differentiable trajectory with $R(\qq(0),0) = \II$ that can be used to steer the behavior of $\qq^{\circ}(t)$
without adversely affecting the cost.

The constant vector ${\bf c}$ provides sufficient freedom to match arbitrary $\dot{\qq}(0)$ or $\qq(1)$,
but not both.\footnote{\blue As is the case in solving the true variational problem, here too in the case of the ansatz
it is not possible to specify both $\dot{\qq}(0)$ and $\qq(1)$ independently. But if a family of solutions parameterized by
varying $\dot{\qq}(0)$ is generated, then a solution within this family that matches the condition
$\qq(1)$ can be chosen.}
For example, when specifying initial conditions,
\beq
\qq^{\circ}(t)  = \qq(0) + \left(\int_{0}^{t} M^{-\half}(\qq^{\circ}(\tau)) R(\qq^{\circ}(\tau),\tau) \, d\tau \right)
\, M^{\half}(\qq(0)) \, \dot{\qq}(0) \,.
\label{mhalf22}
\eeq
{\blue from this $\dot{\qq}(0)$ can be selected so as to match $\qq(1)$ when the matrix in parenthesis is invertible.
Of course, it is not possible to specify both $\dot{\qq}(0)$ and $\qq(1)$ independently.}

In general this will not be a solution to the E-L equations for the variational problem with integrand (\ref{kedef}), unless $M(\qq) = M_0$ is constant. And as such the solution will not be optimal in general.
But since the E-L solution has $T = T_0 = const)$, and
$$  \half (\dot{\qq}^{\circ})^T M(\qq^{\circ}) \dot{\qq}^{\circ} \,=\,
\half \dot{\qq}(0)^T M(\qq(0)) \dot{\qq}(0) \,=\, T_0 \,,$$
the cost of this ansatz appears not to be so bad. Moreover its structure lends itself to rapid numerical shooting to match the distal boundary condition.

For example, consider the case of geodesics in the Poincar\'{e} half plane model. Suppose that the solution were not known and we seek a reasonable approximation. Using this ansatz with $R = \II$, the result would be
$$ \dot{x}_1^{\circ} \,=\, c_1 {x}_2^{\circ} \,\,\,{\rm and}\,\,\, \dot{x}_2^{\circ} \,=\, c_2 {x}_2^{\circ} \,. $$
This means that
$$ {x}_2^{\circ}(t) = {x}_2(0) e^{c_2 t} $$
which can be back substituted to give
$$ {x}_1^{\circ}(t) = \alpha + \beta e^{c_2 t} \,. $$
Now suppose we want to compare to the exact solution presented in the previous section.
One way to compare would be to examine how different the trajectory with the same initial conditions is.
However, a more telling comparison is to examine the true solution and ansatz approximation for given
boundary conditions. The above solutions are written in terms of boundary conditions as
\beq
{x}_2(t) \,=\, {x}_2(0) e^{t \ln({x}_2(1)/{x}_2(0))}
\eeq
which can be back substituted to give
\beq
{x}_1^{\circ}(t) = \left(x_1(0) - x_2(0) \frac{x_1(1) - x_1(0)}{x_2(1) - x_2(0)}\right) + x_2(0) \frac{x_1(1) - x_1(0)}{x_2(1) - x_2(0)} e^{t \ln({x}_2(1)/{x}_2(0))} \,.
\eeq

Consider the boundary conditions $\xx(0)$ and $\xx(1)$ from the solutions in the previous
section:
$$ \xx^*(0) \,=\,
\left(\begin{array}{c}
(bd + ac)(d^2 + c^2)^{-1} \\
(d^2 + c^2)^{-1}
\end{array}\right)
\,\,\,{\rm and}\,\,\,
\xx^*(1) \,=\,
\left(\begin{array}{c}
(bd + ac e^2)(d^2 + c^2 e^{2})^{-1} \\
e (d^2 + c^2 e^{2})^{-1}
\end{array}\right) \,. $$

%*** check ***
%$$
%x_1^*(t) = (bd + ac e^{2t})(d^2 + c^2 e^{2t})^{-1} \,\,\,{\rm and}\,\,\,
%x_2^*(t) = e^{t} (d^2 + c^2 e^{2t})^{-1} \,.
%$$

Though it meets the boundary condtions,
the energy for the ansatz is a larger constant than for the geodesics.
Introducing the rotation $R(x_1,x_2,t)$ provides freedom to reduce this energy while still meeting
the same boundary conditions. For example, suppose $R = R(t)$ defined by counterclockwise measured angle $\theta(t)$. Then
$$ \frac{\dot{x}_2^{\circ}}{{x}_2^{\circ}} = c_1 \sin\theta(t) + c_2 \cos\theta(t) $$
or
$$ {x}_2^{\circ}(t) \,=\, {x}_2^{\circ}(0) \, \exp\left(c_1 \int_{0}^{t} \sin\theta(\tau) d\tau + c_2
\int_{0}^{t} \cos\theta(\tau) d\tau \right)\,. $$
In the simplest case the rotation has one variable to use to steer the solution: $\theta(t) = \omega_0 t$,
and ${x}_1^{\circ}(t)$ can be obtained numerically.\

\section{Global Optimality and Bootstrapping} \label{bootstrapsec}

In the general {\blue vectorial} case there are many possible paths connecting the initial and final values
$\qq(0)$ and $\qq(1)$. For this reason the argument used in the 1D case earlier cannot generalize and the Euler-Lagrange equation
does not in general provide a globally optimal solution. The {\blue vectorial} extension of the integrand
in (\ref{Jy}) is of the form $\dot{\qq}^T M(\qq) \dot{\qq}$. Here
$M(\qq)$ might be the metric tensor for a Riemannian manifold, in which case the solutions of the
Euler-Lagrange equations will be equivalent to the equations giving geodesic paths. As is clear even in the
simple cases such as the spheres $S^1$ and $S^2$, the geodesics connecting any two points are not unique.
Moreover, both the shorter or the longer of the two great arcs connecting
two points are solutions to the variational problem when the points are not antipodal. And when they are,
an infinite number of great arcs connect them, illustrating that shortest paths need not be unique.

From these examples it is clear that global optimality is closely related to the statement of
the problem as a boundary value problem. If instead the problem is restated as that of finding
the path from $\qq(0)$ with specified value of $\dot{\qq}(0)$ that minimizes
the cost functional $\int_{0}^{1} f(\qq,\dot{\qq},t) \,dt$,
then the nonuniqueness of solutions of the Euler-Lagrange equations disappears.
In the case of the sphere, even the longer great arc connecting two points then will be globally optimal.
In the one-dimensional case with cost of the form
$m(y) (y')^2$ on a domain that is an interval, the paths had only one possible direction to go
and so the distinction between
the initial-value and boundary-value problems was not important.
While this is not true in general {\blue for functionals of multi-dimensional vectorial paths $\qq(t) \in \IR^N$},
special cases still do exist in which global
optimality can be proved. For example, when
$M(\qq) = M_0$ is constant, or when $m_{ij}(\qq) = m_{i}(q_i) \delta_{ij}$, then when the domain is a product of
intervals so that $\qq \in [0,1]^n$, then the resulting boundary value problem
will be globally minimized by the solution of the Euler-Lagrange equations. Global optimality is verified
by evaluating the cost of $\qq(t) = \qq^*(t) + \bfepsilon(t)$ where $\bfepsilon(0) = \bfepsilon(1) = {\bf 0}$. As with the case of the line,
the effect of adding $\bfepsilon(t)$ is that the cost will only increase relative to that for $\qq^*(t)$.

\subsection{Bootstrapping Globally Optimal Euler-Lagrange Solutions}

Let
$$ J_0[\bftheta,\qq] \,=\, \int_{0}^{1} f_0(\bftheta,\dot{\bftheta},\qq,\dot{\qq},t) dt $$
where
\beq
f_0(\bftheta,\dot{\bftheta},\qq,\dot{\qq},t) \,\doteq\,
\half \|\dot{\bftheta} - A(\qq,t)\dot{\qq}\|_W^2 \,,
\label{f1kdkhsnew}
\eeq
$\qq \in  \mathbb{R}^N$, ${\bftheta} \in \mathbb{R}^P$, and
$\|B\|_W^2 = {\rm tr}(B^T W B)$ is the square of the weighted Frobenius norm with the weighting defined by
a given constant positive definite symmetric matrix
$W=W^T >0$.

If $\qq(t)$ is fixed in advance, then the globally minimal solution for $\bftheta(t)$ is obviously defined by
$\dot{\bftheta} = A(\qq,t)\dot{\qq}$. This is true
regardless of whether E-L regularity conditions are imposed on $A$ or not.

Alternatively, if boundary conditions on $\bftheta$ are imposed and if $A$ does satisfy E-L regularity conditions
(e.g., $A:\mathbb{R}^N \times [0,1] \,\longrightarrow\, \IR^{P \times N}$ is twice continuously differentiable),
then the E-L equations in $\bftheta$ give
$$ \frac{d}{dt} \left\{W\left[\dot{\bftheta} - A(\qq,t)\dot{\qq}\right]\right\} \,=\, {\bf 0} $$
or
\beq
\dot{\bftheta} -  A(\qq,t) \dot{\qq} \,=\, {\bf a} \,,
\label{first4kn3nd}
\eeq
where the vector ${\bf a}$ is constant. Integrating again gives another constant ${\bf b}$
and the solution is $\bftheta(t) = \bftheta_{[\bf q]}^*(t \,|\, {\bf a},{\bf b})$ where
\beq
\bftheta_{[\bf q]}^*(t \,|\, {\bf a},{\bf b}) \,\doteq\, {\bf a}t + {\bf b} + \int_{0}^{t} A(\qq(\tau),\tau) \dot{\qq}(\tau) d\tau \,.
\label{first4kn3ndaa}
\eeq
Here $*$ denotes that this is the extremal solution the functional $J_0[\bftheta,\qq]$, the subscript denotes that this solution is a functional of $\qq$, and the solution is conditioned on boundary conditions specified by
${\bf a}$ and ${\bf b}$. When no boundary conditions are imposed, $\bftheta_{[\bf q]}^*(t \,|\, {\bf 0},{\bf b})$
makes $f_0 = 0$, which is the global minimum value and is independent of how the differentiable differentiable $\qq(t)$ is chosen.
When ${\bf a} \neq {\bf 0}$, $\bftheta_{[{\bf q}]}^*(t \,|\, {\bf a},{\bf b})$ has sufficient
freedom to match boundary conditions and for any fixed given ${\bf q}(t)$ this choice of $\bftheta(t)$ is again the global minimum
since a perturbation of the form $\bftheta_{[{\bf q}]}^*(t \,|\, {\bf a},{\bf b}) \,\longrightarrow\,
\bftheta_{[{\bf q}]}^*(t \,|\, {\bf a},{\bf b}) + \bfepsilon(t)$ with $\bfepsilon(0) = \bfepsilon(1) = {\bf 0}$ will only increase
the value of the quadratic functional. But when ${\bf a} \neq {\bf 0}$, the choice of $\qq(t)$ is not completely free
for reasons explained below.

Assuming sufficient regularity,
\beq
\frac{d}{dt}\left(\frac{\partial f_0}{\partial \dot{q}_i}\right) - \frac{\partial f_0}{\partial {q_i}} \,=\, \\
- \left(\frac{dA}{dt} {\bf e}_i\right)^T W(\dot{\bftheta} - A\dot{\qq})
- (A {\bf e}_i)^T W \frac{d}{dt}(\dot{\bftheta} - A\dot{\qq})
+ \left(\frac{\partial A}{\partial q_i} \dot{\qq}\right)^T W(\dot{\bftheta} - A\dot{\qq}) \,.
\label{first3kn2nd}
\eeq
Substituting in (\ref{first4kn3nd}) gives
\beq
\frac{d}{dt}\left(\frac{\partial f_0}{\partial \dot{q}_i}\right) - \frac{\partial f_0}{\partial {q_i}} \,=\, \\
- \left(\frac{dA}{dt} {\bf e}_i\right)^T W {\bf a} \,+\,
\left(\frac{\partial A}{\partial q_i} \dot{\qq}\right)^T W {\bf a} \,.
\label{jbccffe}
\eeq
Consequently, if ${\bf a} = {\bf 0}$, these equations vanish. Otherwise, equating the above expression to zero
provides a constraint on $\qq(t)$.

That said, the goal in these computations is not to derive E-L Equations in $\qq$
since $J_0[\bftheta,\qq]$ is a degenerate variational problem (in the sense
that the mass matrix is singular). Rather, the above will be used in
the ``bootstrapping'' procedure described below.

Now suppose that
\beq
f_1(\qq,\dot{\qq},t) = \half \dot{\qq}^T M(\qq) \dot{\qq}
\label{f1kdkhs}
\eeq
is given for which the corresponding Euler-Lagrange equations globally solve a variational problem
with $\qq(0)$ and $\qq(1)$ specified. Here it is shown that
this solution then can be used as a seed to generate
a globally optimal solution to a larger variational problem. This is referred to here as ``bootstrapping''.
Define
\beq f_2(\bftheta,\dot{\bftheta},\qq,\dot{\qq},t) = f_1(\qq,\dot{\qq},t)
+ f_0(\bftheta,\dot{\bftheta},\qq,\dot{\qq},t)
\label{bootst33}
\eeq
For example, $f_1$ could correspond to the situation in Subsection \ref{1Dglobal}.
This scenario is akin to the ``warped product'' defined in \cite{Burago} and ``Kaluza-Klein Lagrangian'' described
in \cite{MarsdenRatiu}, though is more general and does not rely on underlying
geometric properties, though the result has geometric implications.

In the following let
\beq
J_1[\qq] \,\doteq\, \int_{0}^{1} f_1(\qq,\dot{\qq},t) dt \,\,\,{\rm and}\,\,\,
J_2[\bftheta,\qq] \,=\, \int_{0}^{1} f_2(\bftheta,\dot{\bftheta},\qq,\dot{\qq},t) dt \,.
\label{J1def}
\eeq
\begin{thm} \label{boot1thm}
% (where $f_1$ can be of the form (\ref{f1kdkhs}) or more general)
If  $f_1$ is sufficiently regular such that the E-L equations provide extremals, and if $\qq^*$ is the global
minimizing extremal of $J_1$,
then when boundary conditions are not imposed on $\bftheta$, a family of global minimizers of $J_2$ exists
and is of the form $(\qq^*,  \bftheta_{[\bf q^*]}^*(t \,|\, {\bf 0},{\bf b}))$ where
$\qq^*(t)$ is the global minimizer of $J_1$ subject to specified
$\qq(0)$ and $\qq(1)$ and ${\bf b} \in \mathbb{R}^P$ parameterizes the family.
\end{thm}

\begin{proof}
By definition, in component form
\beq
\frac{d}{dt}\left(\frac{\partial f_1}{\partial \dot{q}_i}\right) - \frac{\partial f_1}{\partial {q_i}} \,=\, {0} \,,
\label{f1ljdsc}
\eeq
or explicitly when $f_1(\qq,\dot{\qq},t) = \half \dot{\qq}^T M(\qq) \dot{\qq}$,
$$ {\bf e}_i^T M(\qq) \ddot{\qq} \,+\, \sum_k \dot{q}_k {\bf e}_i^T \frac{\partial M}{\partial q_k} \dot{\qq}
- \half \dot{\qq}^T \frac{\partial M}{\partial q_i} \dot{\qq} \,=\, {0}\,, $$
(though restricting to this form is not required in the proof).

Let the solution to the system of equations (\ref{f1ljdsc}) subject to boundary conditions $\qq(0)=\qq_0$ and $\qq(1)=\qq_1$ be
denoted as $\qq^*(t)$.

On the other hand, evaluating
$$
\frac{d}{dt}\left(\frac{\partial f_2}{\partial \dot{\theta}_i}\right) - \frac{\partial f_2}{\partial {\theta_i}} \,=\, {0} $$
for all values of $i$
simply results in (\ref{first4kn3ndaa}).

The Euler-Lagrange equations for the new system corresponding to coordinates $\{q_i\}$, written in component form, will be
$$ \frac{d}{dt}\left(\frac{\partial f_2}{\partial \dot{q}_i}\right) - \frac{\partial f_2}{\partial {q_i}} \,=\, {0} . $$
From this it is clear that all of the $f_1$ terms disappear when evaluated at $\qq = \qq^*$, and
if ${\bf a} = {\bf 0}$, then (\ref{jbccffe}) gives that $\bftheta_{[\bf q^*]}^*(t \,|\, {\bf 0},{\bf b})$
causes all other terms to vanish. This means that $(\qq^*(t), \bftheta_{[\bf q^*]}^*(t \,|\, {\bf 0},{\bf b})$  is
an E-L extremizer of $J_2$.

The global minimality of the solutions $(\qq^*(t), \bftheta_{[\bf q^*]}^*(t \,|\, {\bf 0},{\bf b})$ is guaranteed by the postulated condition that
$\qq^*(t)$ is obtained a priori, and the global minimality of $\bftheta_{[\bf q^*]}^*(t \,|\, {\bf 0},{\bf b})$ in (\ref{first4kn3ndaa}) as a solution to
a quadratic cost problem. That is, if $\qq^*(t) \,\rightarrow\, \qq^*(t) + \bfepsilon_1(t)$ it will only increase the cost of $J_1$ without
affecting the cost of $J_0$, since $\bftheta^*(t)$ is keyed to changes in $\qq(t)$. Therefore, if
$\bftheta(t) = \bftheta_{[\bf q^* + \bfepsilon_1]}^*(t \,|\, {\bf 0},{\bf b}) + \bfepsilon_2(t)$ is substituted
into the cost function, $\bfepsilon_1(t)$ will have no effect on $J_0$, which will remain zero, and the addition of
$\bfepsilon_2(t)$ only increases the cost from zero to $\int_{0}^{1} \|\dot{\bfepsilon}\|_W^2 dt$.
% $\bfepsilon(0) = \bfepsilon(1) = {\bf 0}$
This does not depend on $\|\dot{\bfepsilon}\|_W$ being small. (This is the same as the argument in (\ref{epsiloneffect}) but with ${\bf a}={\bf 0}$.)
That is, the $f_0$ term in $f_2$ is uniquely zero when ${\bf a} = {\bf 0}$, and there is no way to
do better.
\end{proof}

Note that when ${\bf a} = {\bf 0}$ it is only possibly to specify boundary conditions on $\bftheta^{\,*}$ at one end.
If ${\bf a} \neq {\bf 0}$ more freedom is allowed, but in general ${\bf q}^*$ will no longer satisfy the E-L equations for
$J_2$ due to the nonzero terms contributed from $J_0$. However, in some special cases those $J_0$ terms vanish, as described below.

\begin{thm} \label{boot2thm}
Two special cases for which global minimality is guaranteed without making the restriction ${\bf a} = {\bf 0}$ are:
(1) When $A(\qq,t) = A_0$ is constant; (2)
When $N = 1$ and $A(q,t) = {\bf a}_1(q) \in \IR^{P \times 1}$ with $q$ being scalar.
The globally minimal solutions in these two cases are respectively
$({\bf q}^*, \bftheta_{{\bf q}^*}^{*}(t\,|\,{\bf a},{\bf b}))$ and $(q^*,\bftheta_{[q^*]}^{*}(t\,|\,{\bf a},{\bf b}))$
where\footnote{In (\ref{thopteq53210ddd}) the subscript ${\bf q}^*$ is used instead of $[{\bf q}^*]$ because in this case
$\bftheta_{{\bf q}^*}^{*}$ depends on ${\bf q}^*$ as a function rather than as a functional.}
 \beq \bftheta_{{\bf q}^*}^{*}(t\,|\,{\bf a},{\bf b}) \,\doteq\, {\bf a}t + {\bf b} + A_0\left({\bf q}^*(t) -
 {\bf q}^*(0)\right) \label{thopteq53210ddd} \eeq
 and
\beq \bftheta_{[q^*]}^{*}(t\,|\,{\bf a},{\bf b}) \,\doteq\, {\bf a}t + {\bf b} +
\int_{0}^{t} {\bf a}_1(q^*(s)) \dot{q}^*(s) \,ds \label{thopteq53210} \eeq
satisfies the expanded E-L equations corresponding to
(\ref{bootst33}) and can satisfy all boundary conditions $\bftheta(0)$ and $\bftheta(1)$ in addition to those in
$\qq$ or $q$.
\end{thm}

\begin{proof}
If ${\bf a} \neq {\bf 0}$ in (\ref{first4kn3nd}) then the necessary conditions established by the E-L equations become
(\ref{jbccffe}).
If $A(\qq) = A_0$ is constant, then the terms in (\ref{jbccffe}) vanish and the case 1 solution results.
This condition also can be satisfied if the dimension of $\qq$ is 1, wherein both terms in parenthesis in (\ref{jbccffe}) become the same and hence cancel.
In this second case, denote $A = {\bf a}_1$, since it
consists of one column vector, gives (\ref{thopteq53210}).
\end{proof}
Global optimality is ensured as before, by introducing a perturbation $\bfepsilon(t)$ of any size
satisifying the boundary conditions $\bfepsilon(0) = \bfepsilon(1) = {\bf 0}$.

The class of problems addressed in this section can be viewed in a slightly different way when (\ref{f1kdkhs}) holds by rewriting (\ref{bootst33}) as
\beq
f(\qq,\bftheta,\dot{\qq},\dot{\bftheta},t) = \half
\left[\begin{array}{c}
\dot{\qq} \\ \\
\dot{\bftheta} \end{array}\right]^T
\left(\begin{array}{cc}
M(\qq) \,+\, A^T(\qq,t) W A(\qq,t) & A^T(\qq,t) W \\ \\
W A(\qq,t) & W \end{array}\right)
\left[\begin{array}{c}
\dot{\qq} \\ \\
\dot{\bftheta} \end{array}\right].
\label{boot33dd}
\eeq
Consequently, a quadratic cost that can be decomposed in this way will have solutions that
inherit the global optimality from the original smaller problem under appropriate boundary conditions.
It also means that the above reasoning can be used recursively. If the matrix in (\ref{boot33dd}) is called $M'(\qq)$ and 
we define $\qq' \doteq [\qq^T, \bftheta^T]^T$, then the problem can be expanded further by adding an additional terms such
as $\|\dot{\bfphi} - A'(\qq',t) \dot{\qq}'\|_{W'}^{2}$.

\subsection{Bootstrapping Globally Optimal Solutions That Lack Regularity}

If the functional $f_1$ or the resulting global optimizers lack sufficient
regularity for the E-L equations to be applicable, the bootstrapping process
can nevertheless be used in some cases to generate globally optimal solutions
on higher dimensional spaces.

For example, if $f_1(q,\dot{q}) = m(q) \dot{q}^2$ and $m$ is as in Theorem 2.2,
then the globally minimal solution of the variational problem with functional
$$ f_2(q,\dot{q}, \theta, \dot{\theta}) = m(q) \dot{q}^2
+ [\dot{\theta} - a(q,t) \dot{q}]^2 $$
with $t \in [0,1]$
and $a:\IR \times [0,1] \rightarrow \IR$ continuous and bounded
will be $q^*(t)$ from Theorem 2.2, and $\theta^*(t) = a_0t + b_0 + \int_{0}^{t} a(q^*(s),s)\dot{q}^*(s)ds$.
This is true even if $m$ is not differentiable, and hence lacks the regularity to use the E-L equations.
Moreover, in cases such as this where $q$ and $\theta$ are both one-dimensional, when $a=a(q)$ is autonomous,
the integral defining $\theta^*$ can be rewritten as
\beq
\theta^*(t) = a_0t + b_0 + \int_{q^*(0)}^{q^*(t)} a(q) dq \,,
\eeq
making it an explicit function of $q^*(t)$.

As a second example, we can bootstrap off of the Mani\'{a} example of the Lavrentiev phenomenon in (\ref{maniaex})
and seek a global minimum of
\begin{equation}
J[x,\theta] = \int_{0}^{1} \left\{
(x^3 - t)^2 \dot{x}^6 + (\dot{\theta} - a(x,t) \dot{x})^2 \right\} \,dt
\label{lavboot}
\end{equation}
subject to boundary conditions $x(0) = 0$ and $x(1) = 1$.
We can simply write $x^*(t) = t^{1/3}$ and $\theta^*(t) = a_0 t + b_0 + \int_{0}^{t}  a(x^*(s),s) \dot{x}^*(s) ds$
for a wide variety of functions $a$, as long as the integral defining $\theta^*(t)$ exists and is finite.

That said, most of the applications that follow involve bootstrapping in the
case when the paths and functional are regular.

\section{Examples from the Theory of Framed Space Curves}

In the classical theory of curves in three-dimensional space, the Frenet-Serret Apparatus is used
to frame arclength-parametrized curves. In the first subsection of this section, a brief review of
that classical theory is given, adapted from Chapter 5 of the author's book \cite{harmonic}. This material
is classical (dating back 150 years) and can be found stated using other notation in \cite{Millman,DoCarmo}.
The material is presented here primarily because subtle differences in the notation used here relative
to the standard presentation will be convenient later.

{\blue

\subsection{The Frenet-Serret Apparatus}

Without loss of generality, every curve can be parameterized by its arclength, $s$.
Associated with every three-times-differentiable arclength-parameterized curve, ${\bf x}(s) \in \IR^3$, are two functions:
{\it curvature}, $\kappa(s)$, and {\it torsion}, $\tau(s)$. These functions
completely describe the shape of the curve up to arbitrary rigid-body motion.

As long as $\kappa(s) \neq 0$, a unique Frenet frame can be attached at every point on the curve.
The orientation of this frame relative to the world frame will be
$$ {R}_{FS}(s) \,\doteq\, [{\bf t}(s), {\bf n}(s), {\bf b}(s)] \in SO(3) $$
where ${\bf t}(s)$, ${\bf n}(s)$, and ${\bf b}(s)$ are respectively the unit tangent, normal, and binormal.
The Frenet frames, the curvature, and the torsion, together are known as the Frenet-Serret apparatus.

The way that the orientation of a Frenet frame evolves along the curve is tied to the curvature and torsion
through the system of first order differential equations\footnote{In most presentations the vectors ${\bf t}$,
${\bf n}$, ${\bf b}$ are stacked rather than being written side-by side, resulting in an equation deceptively
similar in appearance to (\ref{fsewneffewc2}) but without introducing ${R}_{FS}$.}
\begin{equation}
\frac{d{R}_{FS}}{ds} = - {R}_{FS} {\Omega_{FS} }
\label{fsewneffewc2}
\end{equation}
where ${\Omega_{FS}}$ is a special skew-symmetric matrix with only two independent nonzero entries
$${\Omega_{FS}(s) } \,=\,
\left(\begin{array}{ccc}
0 & \kappa(s) & 0 \\
-\kappa(s) & 0 & \tau(s) \\
0 & -\tau(s) & 0 \end{array} \right) \,. $$
The Frenet frame at each value of arclength is then the rotation-translation pair
$({R}_{FS}(s), {\bf x}(s)) \in SO(3) \times \IR^3$.
No group law is defined. Later in this article several possibilities are explored.
}

\subsection{Minimally Varying Frames}

The classical Frenet frames are not the only way to attach reference frames to space curves,
as pointed out by Bishop \cite{6bishop}.
Moroever, the Frenet frames are not optimal in the sense of twisting as little as possible around
the tangent as the curve is traversed. The word `twist' is somewhat problematic, as it can be
confused with torsion. Whereas torsion is a twisting of the curve, the twisting of the frame attached
to the curve does not affect the curve shape. {\blue For this reason, the word ``roll'' can be used in place
of ``twist'', which is consistent with the terminology used to describe motions of aircraft and ships.}

Two ways to modify Frenet frames are described here, which demonstrate the variational formulation from
earlier in this paper. First the optimal amount of roll along the tangent relative to the Frenet frames
is established. Next, alternative parameterizations in place of arclength of the curve that have optimality
properties are examined. Finally, the optimal set of frames and reparameterization are combined, which represents
an example of the bootstrapping described previously.

Prior to conceiving the general concept of bootstrapping globally optimal variational problems, the author
considered special cases that arose from the field of robotics \cite{mythesis,optimalpaper}, as reviewed in
\cite{harmonic,stochastic}, from which some of this presentation is adapted.

\subsubsection{Minimally Rolling Frames}

Consider the Frenet frames as a start, and allow freedom to roll around the tangent so that the new frames have
orientation
$$ R = R_{FS}\, R_1(\theta)\,\,\,{\rm where}\,\,\, R_1(\theta) \,=\, \left(\begin{array}{ccc}
1 & 0 & 0 \\
0 & \cos\theta &-\sin\theta \\
0 & \sin\theta & \cos\theta \end{array} \right) \,.$$
The optimal $\theta(s)$ will minimize the functional
\begin{equation}
J = \half \int_{0}^{1} {\rm tr}\left(\frac{dR}{ds} \frac{dR}{ds}^{T}\right)\,ds.
\label{cost11}
\end{equation}
Computing the derivatives and expanding out gives the explicit expression
\beq
\half {\rm tr}\left(\frac{dR}{ds} {\frac{dR}{ds}}^{T}\right) = \kappa^2 + \tau^2
+ 2 \frac{d\theta}{ds} {\bf e}_1 \cdot {\bfomega}_{FS} +
\left(\frac{d\theta}{ds}\right)^2 \label{ksndsdkcc} \eeq
where
$${\bfomega}_{FS} \,\doteq\, \left(\begin{array}{c}
\tau \\
0 \\
\kappa \end{array} \right). $$
This is obtained by using the properties of the trace. For example, ${\rm tr}\left(\frac{dR}{ds} {\frac{dR}{ds}}^{T}\right) =
{\rm tr}\left(R^T \frac{dR}{ds} {\frac{dR}{ds}}^{T} R\right)$ which is the trace of the product of two
skew-symmetric matrices. And
for skew-symmetric matrices $\Omega_1$ and $\Omega_2$
$$ \half {\rm tr}\left(\Omega_1 \Omega_2 \right) = {\bfomega}_1 \cdot {\bfomega}_2 $$
where $\bfomega_i$ is the unique vector such that
$$ \bfomega_i \times {\bf v} \,=\, \Omega_i {\bf v} $$
for every ${\bf v} \in \IR^3$. This relationship between $\bfomega_i$ and $\Omega_i$ is denoted as
\beq
{\bfomega}_i \,=\, (\Omega_i)^{\vee} \,\,\,{\rm and}\,\,\, \Omega_i \,=\, \hat{\bfomega}_i\,.
\label{hatveedef}
\eeq

From (\ref{ksndsdkcc}) and the expression for ${\bfomega}_{FS}$ it is clear that
\begin{equation}
\half \,{\rm tr}\left(\frac{dR}{ds} {\frac{dR}{ds}}^{T}\right) = \kappa^2 + \left(\tau
+ \frac{d\theta}{ds}\right)^2,
\label{costtwist}
\end{equation}
and the cost functional is minimized when
$$ \frac{d\theta}{ds} = -\tau. $$
Though this had been known for many decades as explained by Bishop \cite{6bishop}, it became known to
the author as part of studies on snakelike robot arms \cite{mythesis, optimalpaper}. In that work the goal
was to use a ``backbone curve'' to capture the overall shape of a physical robot made from joints and links, and
since each section of the robot was commanded to adhere to a section of the curve, by limiting the amount of change in position and orientation between reference frames at the end of each section of curve it was easier
for the robot to conform.

The above defines frames that have globally minimal roll. In some applications it is required
to satisfy conditions at the end points. Such solutions will be of the form
$$ \theta^*(s) = c_1 + c_2 s - \int_{0}^{s} \tau(\sigma)\,d\sigma $$
where $c_1$ and $c_2$ are determined by the end conditions
 $\theta(0) = \theta_0$ and $\theta(1) = \theta_1$.

 The fact that this is globally optimal follows from the same arguments as in the example of the  line (i.e., let $\theta(s) = \theta^*(s) + \epsilon(s)$ and show that this can only make things worse).

\subsubsection{Optimal Reparameterization for Least Variation}

Within the same motivating snakelike robot application, it was reasoned that
if the spacing along the curve was relaxed from being arclength, then this
would provide additional freedom for the physical robot to conform to the backbone curve.
In analogy with how the Frenet frames are not the only (or perhaps not even the most desirable) orientations
to attach to a curve, arclength may not be the best parameter to use.
In optimal reparameterization, the idea is to no longer use
the arclength, $s$, but instead use a new curve parameter $t$
such that $s(t)$ satisfies $s(0) =0$ and $s(1) = 1$ such that a
cost functional of the form
\index{reparameterization}
$$ J = \int_{0}^{1} \left\{\half \,r^2\, {\rm tr}\left(\frac{d R(s(t))}{dt}
\frac{d R^{T}(s(t))}{dt}\right) + \left(\frac{d s}{dt}\right)^2 \right\}\,dt $$
is minimized. The constant $r$ is used to introduce units of length because the units to evaluate
bending are measured in angular units whereas extension is measured in units of length.
For example, if a ball of radius $r$ and unit mass is traversing this trajectory, $J$ is a measure of the total rotational and translational kinetic energy of the ball integrated over the trajectory. If it speeds up or slows down more quickly than the optimal solution, this cost will increase. This idea will be generalized later in the paper.

The integrand in the above functional is of the form
$ m(s) ({s'})^2$ where $s' = ds/dt$  and
$$ m(s) = \half \,r^2\, {\rm tr}\left(\frac{d R}{ds}
\frac{d R^{T}}{ds}\right) + 1, $$
and therefore the 1D globally optimal solution discussed previously in Theorem \ref{1Dglobthm} applies.
Call this solution $s^*(t)$.

\subsection{Simultaneous Optimal Reparameterization and Roll Modification}

In the physical robotics problem discussed earlier, it was desirable to combine the optimal roll and curve reparametrization results. This can be done by first solving for the optimal
roll and setting $R(s)$, followed
by performing the optimal curve reparametrization, resulting in the frame $(R_{FS}(s^*(t)) R_1(\theta^*(s^*(t))),{\bf x}(s^*(t))$. A natural question to ask is if this composition of optimal solutions to subproblems optimally
solves the larger composite problem. In general one would expect composition of optimal solutions to smaller
problems to be suboptimal on the larger space. But in the case of optimal curve reparametrization and rolling,
the directions are in a sense orthogonal and commuting (local rolling and extension about/along the same vector).

Moreover, this problem of simultaneous curve reparameterization and optimal
roll distribution serves as a concrete application of the class of bootstrapping problems discussed earlier in this paper. Start with a curve initially parameterized by arclength,
$\xx(s)$ for $s \in [0,1]$, and with the Frenet frames $[\xx(s), R_{FS}(s)]$. We seek
a new set of smoothly evolving reference frames of the form
$[\xx(t), R(t)] = [\xx(s(t)), R_{FS}(s(t)) R_1(\theta(s(t)))]$, where $R_1(\theta)$
is an added roll, of the Frenet frames about the tangent together with a reparameterization $s=s(t)$.

A cost function in the two variables $(s(t),\theta(t))$ can be constructed of the form
\beaq
C &\doteq& \half \int_{0}^{1} \left\{\half r^2 {\rm tr}(\dot{R} \dot{R}^T) +
\dot{\xx} \cdot \dot{\xx} \right\} dt \\
&=& \half \int_{0}^{1} \left\{
(r^2 \kappa^2(s) + 1) \dot{s}^2 + r^2(\tau(s) \dot{s} + \dot{\theta})^2
\right\} dt.
\eeaq
This functional is the same form as in (\ref{bootst33}) with $\theta$ taking the place of $\bftheta$ and
$s$ taking the place of $\qq$. If the second term in the integral is zero, then the remaining term is a one-dimensional variational
problem with $f(s,\dot{s},t) = \half m(s) (\dot{s})^2$ for which a globally minimal solution can be obtained.
From the previous discussion about bootstrapping, Theorem \ref{boot2thm} guarantees the global optimality
of the composite problem.

What is also interesting here is that the two problems (optimal roll distribution for an arclength-parameterized curve and optimal curve reparametrization) can be be considered sequentially and then merged, and this merged solution will be the same as the joint globally optimal one.
Moreover, this example serves as a model for far more general problems with these characteristics.

\section{Global Optimality in the Lie Group Setting}

Let $G$ be a real matrix Lie group. The dimension of $G$ as a manifold is denoted as $N$, and the constituent matrices, $g \in G$, are elements of $\IR^{n\times n}$.

The variational calculus problem on matrix Lie groups can be formulated as that of finding paths
$g(t) \in G$ that extremize a functional
\begin{equation}
J = \int_{t_1}^{t_2} f(g;\bfxi;t) \,dt \label{eq:groupfunc}
\end{equation}
where $\bfxi = (g^{-1}\dot{g})^{\vee}$ is a kind of velocity, which is computed simply as the matrix product of
$g^{-1}$ and $\dot{g}$, the latter of which is not an element of $G$. It can be shown that
$g^{-1}\dot{g} \in  {\cal G}$, the Lie algebra corresponding to $G$, which has $N$
basis elements $\{E_i\}$, each of which are in $\IR^{n\times n}$. This means that
$$ g^{-1}\dot{g} \,=\, \sum_{i=1}^{N} \xi_i E_i $$
and a vector $\bfxi = [\xi_1,...,\xi_N]^T \in \IR^N$ can be assigned to each element of ${\cal G}$.
The notation $\vee$ takes elements of ${\cal G}$ and returns vectors in $\IR^N$, and the opposite is
the $\hat{(\cdot)}$ which converts $\bfxi \in \IR^N$ into $\sum_{i=1}^{N} \xi_i E_i \in {\cal G}$.
An example of this was demonstrated for ${\cal G} = so(3)$ in (43).
The choice of $\{E_i\}$ with the definition $E_i^{\vee} = {\bf e}_i$ induces an inner product on ${\cal G}$
$$ (E_i,E_j) = {\bf e}_i \cdot {\bf e}_j = \delta_{ij} $$
and fixes the metric for $G$.

Each $g_i(t) \doteq \exp(t E_i)$ is a curve in $G$ defining one parameter subgroups, where $\exp(\cdot)$ is the matrix exponential. 
These curves are geodesics relative to the metric for which $(E_i,E_j) = \delta_{ij}$.

The necessary conditions for optimality in this setting are the {\it Euler-Poincar\'{e} equations}, which
can be written as \cite{vpoincare,vholm,vbloch,stochastic}
\begin{equation}
\frac{d}{dt}\left(\frac{\partial f}{\partial \xi_i}\right) +
\sum_{j,k=1}^{N} \frac{\partial f}{\partial \xi_k} C_{ij}^{k} \,
\xi_j = \tilde{E}_{i} f.
\label{eq:biggy2}
\end{equation}
where $C_{ij}^{k}$ are the structure constants for the Lie algebra ${\cal G}$ relating the Lie bracket of basis elements as a linear combination of basis elements
$$ [E_i,E_j] = \sum_{k=1}^{N} C_{ij}^{k} E_k \,.$$
For a function $f:G \rightarrow \IR$,
$$ (\tilde{E}_{i} f)(g) \,\doteq\, \left.\frac{d}{ds} f(g \circ \exp(s E_i))\right|_{s=0} $$
is a directional derivative in the direction of the $i^{th}$ one-dimensional subgroup. Throughout this section, $\exp$ denotes
the matrix exponential, which when applied to Lie algebra elements produces Lie group elements. 
(In (\ref{eq:biggy2}) $f:G\times {\cal G} \times \IR_{\geq 0} \,\rightarrow\, \IR$, but the $\tilde{E}_{i}$
is essentially a partial derivative, hence only the dependence on $G$ is pertinent).

A version of the above {Euler-Poincar\'{e}} equation with additional end constraints imposed with Lagrange multipliers can be found in \cite{vkim,rucker} in the context of modeling DNA loops and thin steerable elastic tubes proposed for medical procedures.

\subsection{Globally Minimal Solutions of the Euler-Poincar\'{e} Equation} \label{globalLie}

Consider the integrand of a cost functional
$$ f(g,\bfxi,t) = \half (\bfxi - {\bf c})^T K (\bfxi - {\bf c}) = \half \sum_{i,j=1}^{N} K_{ij} (\xi_i - c_i)  (\xi_j - c_j) $$
with constant $K = K^T >0$, the absolute minimal value is obtained when $\bfxi = {\bf c}$. The corresponding globally minimal trajectory on the group is then defined by
$$ g^{-1} \dot{g} \,=\, \sum_{i=1}^{N} c_i E_i \,. $$
If ${\bf c}$ is a constant vector, then
$$ g(t) \,=\, g(0) \exp\left(t \sum_{i=1}^{N} c_i E_i\right) \,. $$
However, this solution is not able to satisfy arbitrary distal end condtions, such as specifying $g(1)$.
If the goal is to do that, then
the Euler-Poincar\'{e} equations of the form
\beq \frac{d}{dt}\left( \sum_{j=1}^{N} K_{ij} (\xi_j - c_j) \right) +
\sum_{j,k=1}^{N} \left( \sum_{l=1}^{N} K_{kl} (\xi_l - c_l) \right) C_{ij}^{k} \xi_j = 0
\label{ep43j3r}
\eeq
need to be satisfied. In general, this requires numerical solution, but in some cases things simplify.
In particular, consider the case when ${\bf c} = {\bf 0}$ and let
\beq S_{lj}^{i} \doteq \sum_{k=1}^{N} K_{kl} C_{ij}^{k}. \label{slji} \eeq
If $S_{lj}^{i} = -S_{jl}^{i}$ then $\sum_{j,l=1}^{N} S_{lj}^{i} \xi_l \xi_j = 0$
and (\ref{ep43j3r}) reduces to
\beq
K \dot{\bfxi} = {\bf 0} \,\, \Longrightarrow \,\, \bfxi(t)  = \bfxi(0)
\,\, \Longrightarrow \,\, g(t) = g(0) \circ e^{t \hat{\bfxi}(0)}.
\label{disappear}
\eeq
The shortest path connecting $g(0) = g_0$ and $g(1) = g_1$ is then
\beq
{g^*(t) = g_0 \circ \exp(t \cdot \log(g_{0}^{-1} \circ g_1)) }
\label{geopath}
\eeq
where $\log$ is the matrix logarithm.

However, if $S_{lj}^{i} \neq -S_{jl}^{i}$ then the path generated by solving the E-P
variational equations is generally not this geometric one. For example,
when $G=SO(3)$ and when the moment of inertia is not isotropic, the optimal path between rotations
can become very complicated. This is relevant to satellite attitude
reorientation problems, as discussed in \cite{junkins,Markley}.

\subsection{The Rotation Group $SO(3)$ and the Sphere $\mathbb{S}^2$}

\subsubsection{Euler-Poincar\'{e} Equations for $SO(3)$}

For the case when $G=SO(3)$ group elements are $3\times 3$ special orthogonal (rotation) matrices, denoted as $R$, and the Lie algebra elements are skew symmetric matrices. In particular
$R^T \dot{R} = \Omega$ corresponds to the body-fixed description of angular velocity, $\bfomega$, and the notation in (\ref{hatveedef}) is used, i.e., $\bfomega = \Omega^{\vee}$ and $\Omega = \hat{\bfomega}$.

Minimizing a functional of the form
\beq
J = \half \int_{t_1}^{t_2} (\bfomega -\bfomega_0)^T {\cal I} (\bfomega -\bfomega_0) dt
\label{so3functional}
\eeq
arises in several applications. For example, when $\bfomega_0 = {\bf 0}$ this
corresponds to minimizing kinetic energy due to rotation where ${\cal I}$ is the moment
of inertia matrix.
The Euler-Poincar\'{e} equations corresponding to the above functional are
\beq
{\cal I} \dot{\bfomega} \,+\, \bfomega \times {\cal I} (\bfomega -\bfomega_0) \,=\, {\bf 0} \,,
\label{epso3}
\eeq
and these become the classical Euler equations of motion without applied moments when $\bfomega_0 = {\bf 0}$.

Trajectory planning on $SO(3)$ is both of fundamental importance as an example of more general geometric methods
as in \cite{iserles_acta,baillieul,crouch,blochcrouch,Jurdjevic,sussmann,helmke}, as well as practical and elegant engineering and computer graphics applications such as motion interpolation \cite{5parkravani,6ravani}
and spacecraft attitute reorientation \cite{junkins,Markley}. In the subsections that follow, specific variational solutions are examined, and their minimality properties are queried.

\subsubsection{Geodesics on $SO(3)$}

The situation in (\ref{disappear}) occurs for the functional in (\ref{so3functional})
in the special case when ${\cal I} = \mathbb{I}$ and $\bfomega = {\bf 0}$ and
(\ref{geopath}) applies. That is, the geodesic path connecting $R_0,R_1 \in SO(3)$ with
${\rm tr}(R_1^T R_2) > -1$ (so that they are not rotations related by a rotation of angle $\pi$) is
\beq
{R^*(t) = R_0 \circ \exp(t \cdot \log(R_0^T R_1)).}
\label{geopathrot}
\eeq
Furthermore, this path can be shown to be globally optimal in this case
because of the structure of the cost function.
To verify this statement, suppose that $R^*(t)$ is replaced with $R(t) = R^*(t) Q(t)$ where $Q(t)$
is an arbitrary differentiable path in $G$ with end points $Q(0)=Q(1) = e$. Then
$$ R^{T} \dot{R} \,=\, (R^* Q)^{T} \{\dot{R}^* Q + R^* \dot{Q}\} \,=\, Q^{T} \left((R^*)^{T} \dot{R}^*\right) Q + Q^{T} \dot{Q} \,. $$
With $\bfomega \,\doteq\, \left(R^{T} \dot{R}\right)^{\vee}$, $ \bfomega^* \,=\, \left((R^*)^{T} \dot{R}^*\right)^{\vee}$, and $\bfomega_Q \,\doteq\, \left(Q^{T} \dot{Q}\right)^{\vee}$ we have
$$ \bfomega \,=\, Q^T \bfomega^* +  \bfomega_Q \,. $$
Then
$$ \|\bfomega\|^2 = \left\|\bfomega^*\right\|^2 + 2 (\bfomega^*)^T Q \bfomega_Q
+ \left\|\bfomega_Q\right\|^2 $$
Note that $\bfomega^*$ is constant and $Q \bfomega_Q = (\dot{Q} Q^T)^{\vee} = \bfomega_Q^s$ (space-fixed description of angular velocity corresponding to $Q(t)$). Therefore
$$ \int_{0}^{1} (\bfomega^*)^T Q \bfomega_Q dt = (\bfomega^*)^T \int_{0}^{1} \bfomega_Q^s dt \,. $$
For small rotations $Q(t)$ that do not move far from the identity for $t\in[0,1]$, the boundary conditions
$Q(0) = Q(1) = \II$ force
$$  \int_{0}^{1} \bfomega_Q^s dt \,=\, {\bf 0} \,. $$
Moreover, both $\bfomega^*$ and any perturbation $Q$ (whether small or large)
must obey the invariance of the problem. For example, if the start and end points are reversed, we get
$\bfomega^* \rightarrow -\bfomega^*$. And since $Q(0) = Q(1) = \II$, and since $SO(3)$ is a symmetric space, then unless $Q(t)$ executes a full $2\pi$
rotation (which cannot help since the total angle traversed from $R_0$ to $R_1$ is less than $\pi$), symmetry under swapping the roles of the end points indicate that the angular velocity $\bfomega_Q^s$ must integrate to zero. That is, there is no way to do better than $R^*(t)$, and the argument is essentially the same
as for the line in the plane.

\subsubsection{The Case When $\bfomega_0 \,\neq\, 0$}

The E-P equation in (\ref{epso3}) when ${\cal I} = \II$ and $\bfomega_0 \,\neq\, 0$ results in
\beq
\dot{\bfomega} \,-\, \bfomega \times \bfomega_0 \,=\, {\bf 0} \,.
\label{epso3aa}
\eeq
This is a linear system of equations with solution
\beq
\bfomega(t) \,=\, \exp(-t \hat{\bfomega}_0) \bfomega(0)
\,=\, \left(\II - \sin (\|{\bfomega}_0\| t) \frac{\hat{\bfomega}_0}{\|{\bfomega}_0\|}
\,+\, (1-\cos(\|{\bfomega}_0\| t) \frac{\hat{\bfomega}_0^2}{\|{\bfomega}_0\|^2}\right)  \bfomega(0)  \,.
\label{epso3aa1}
\eeq
Since $\exp(t \hat{\bfomega}_0) \, \bfomega_0 \,=\, \bfomega_0$ it follows that
$$ \|\bfomega(t) - \bfomega_0\|^2 \,=\,  \|\bfomega(0) - \bfomega_0\|^2 \,. $$
With initial conditions $R(0) = \II$, the small-time solution for $R(t)$ is then
$$ R(t << 1) = \II + \int_{0}^{t} \hat{\bfomega}(\tau) \, d\tau\,=\,  \II + \hat{\bftheta} $$
where
$$ {\bftheta}(t) \,=\, \int_{0}^{t} {\bfomega}(\tau) \, d\tau\,=\, \left(t \II + [\cos  (\|{\bfomega}_0\| t) -1]\frac{\hat{\bfomega}_0}{\|{\bfomega}_0\|^2}
\,+\, (t|{\bfomega}_0\|-\sin(\|{\bfomega}_0\| t) \frac{\hat{\bfomega}_0^2}{\|{\bfomega}_0\|^3}\right)  \bfomega(0) \,. $$
A small perturbation of this will be $R(t) \exp(\hat{\bfepsilon})$ with $\bfepsilon(0) = \bfepsilon(1) = {\bf 0}$, or
$\bftheta \rightarrow \bftheta + \bfepsilon$, leading to the cost
$\|\bfomega(t) - \bfomega_0 + \dot{\bfepsilon}\|^2$. As with the case of the line in the plane,
this reduces to $\|\bfomega(t) - \bfomega_0\|^2 + \|\dot{\bfepsilon}\|^2$ because
the cross term integrates to zero.

\subsubsection{Geodesics on the Sphere}

A convenient way to describe geodesics connecting two non-colinear points ${\bf a}, {\bf b} \in \mathbb{S}^2$ is via the rotation matrix ${R}({\bf a},{\bf b})$,
which most directly transforms a unit vector ${\bf a}$
into the unit vector ${\bf b}$:
$$ {\bf b} = {R}({\bf a},{\bf b}) {\bf a}. $$
If $\theta_{ab} \in [0,\pi]$ denotes the angle of rotation measured counterclockwise from ${\bf a}$ to ${\bf b}$ around the axis defined by ${\bf a} \times {\bf b}$. Then again using the notation in (\ref{hatveedef}) it can be shown that \cite{harmonic}
\begin{equation}
{R}({\bf a},{\bf b}) = \exp\left(\frac{\theta_{ab}}{\sin\theta_{ab}}\, \widehat{{\bf a} \times {\bf b}}\right)  =
\II + \widehat{{\bf a} \times {\bf b}} \,+\, \frac{(1- {\bf a} \cdot {\bf b})}{\|{\bf a} \times {\bf b}\|^2} \left(\widehat{{\bf a} \times {\bf b}}\right)^2.
\label{transab}
\end{equation}

Then the minimal length geodesic arc connecting ${\bf a}$ and ${\bf b}$ is
$$ {\bf u}(t) = \exp(t \log {R}({\bf a},{\bf b})) {\bf a} = {\bf a}
 + t \left(\widehat{{\bf a} \times {\bf b}}\right) {\bf a} \,+\, t^2 \frac{(1- {\bf a} \cdot {\bf b})}{\|{\bf a} \times {\bf b}\|^2} \left(\widehat{{\bf a} \times {\bf b}}\right)^2 {\bf a} $$
 for $t \in [0,1]$. In the above, the hat notation in (52) is used.

\subsection{Globally Optimal Solutions on Direct and Semi-Direct Products}

Here the concept of $(\theta^*,s^*)$ from optimal curve rolling and reparameterization is generalized
where each variational subproblem is in a subgroup of a larger group, which then are bootstrapped to
obtain solutions on the larger space.

\subsubsection{Direct Products}

Let $G$ and $H$ be a matrix Lie groups with respective Lie algebra basis elements $\{E_i\}$ and
let $\{{\cal E}_j\}$. Given a direct product with elements $(g,h) \in G \times H$ and the direct product group law
$$ (g_1,h_1)(g_2,h_2) = (g_1 g_2, h_1 h_2) \,. $$
Let $\bfomega = (g^{-1} \dot{g})^{\vee}$ and
${\bf v} = (h^{-1} \dot{h})^{\vee}$ denote the velocities. (The reason for this choice of names will become clear soon.)
As an example, the so-called {\it pose change group} \cite{posechange}  is the direct product
$SO(3) \times \IR^3$. This can be described as a matrix Lie group with elements of the form of a direct sum
$$ (R,t) \,=\, R \oplus \left(\begin{array}{cc}
\II & {\bf t} \\
{\bf 0}^T & 1 \end{array}\right)\,\in\, \IR^{7\times7}. $$
In this context
$$ h^{-1} \dot{h} \,=\, \left(\begin{array}{cc}
\II & -{\bf t} \\
{\bf 0}^T & 1 \end{array}\right)
\left(\begin{array}{cc}
\OO & \dot{\bf t} \\
{\bf 0}^T & 0 \end{array}\right) = \left(\begin{array}{cc}
\OO & \dot{\bf t} \\
{\bf 0}^T & 0 \end{array}\right) $$
and ${\bf v} =  (h^{-1} \dot{h})^{\vee} = \dot{\bf t}$.

$SO(3) \times \IR^3$ is \underline{not} the group of handedness preserving Euclidean motions, but it
still is relevant in describing changes in pose. For example, if an aerial vehicle is moving, then it can be
convenient to track its relative orientation and absolute position as described in \cite{quadrotor}. Similarly, for
protein-protein docking this direct product approach is often taken \cite{Kozakov,Vakser,Gray}.

Moreover, if
$(R, {\bf t}) \in SO(3) \times \IR^3$, it is convenient to define body-fixed angular velocity
$\bfomega = (R^T \dot{R})^{\vee}$ and the space fixed view of translational velocity ${\bf v} = \dot{\bf t}$.
Indeed, in the classical expression for kinetic energy
$$ T = \half \bfomega^T {\cal I} \bfomega \,+\, \half m {\bf v} \cdot {\bf v} $$
this is exactly what is done.

It is clear that if the globally minimal solutions to
$$ J_1 = \int_{t_1}^{t_2} f_1(g,\bfomega,t) dt \,\,\,{\rm and}\,\,\, J_2 = \int_{t_1}^{t_2} f_2(h,{\bf v},t) dt $$
are obtained separately, then the globally minimal solution for $J_1+J_2$ on the product group will be obtained by pairing the solutions for each independently as
$(g^*(t), h^*(t))$. In the case of the direct product $SO(3) \times \IR^3$
\beq
(R^*(\tau),{\bf t}^*(\tau)) = \left(R_0 \circ \exp(\tau \cdot \log(R_0^T R_1))\,,\, {\bf t}_0 + ({\bf t}_1 - {\bf t}_0)\tau \right)\,.
\label{geopathdirectprod}
\eeq
Slightly more interesting than that is when the cost is of the form\footnote{The case
$f(g,h,\bfomega,{\bf v}) = f_1(h,{\bf v})\,+\,\|\bfomega - A(h) {\bf v}\|_W^2$ can be handled analogously.}
\beq
f(g,h,\bfomega,{\bf v}) = f_1(g,\bfomega)\,+\,\frac{1}{2} \|{\bf v} - A(g) \bfomega\|_W^2
\label{twisted1}
\eeq
in analogy with the coordinate-dependent discussion of bootstrapping earlier in the paper.
Assuming that $f_1$ satisfies the E-P equation, then the E-P equation for $g$ will reduce to
$$  \cdot \frac{d}{dt} {\bf e}^T A^TW(A \bfomega - {\bf v}) - \bfomega \times (A \bfomega - {\bf v}) =
(A \bfomega - {\bf v}) \tilde{E}_i A \bfomega \,.
$$
and the E-P equation for $h$ will reduce to
$$  \frac{d}{dt} (A \bfomega - {\bf v}) = {\bf 0}\,. $$

The condition for global minimality
$$ {\bf v} = A(g) \bfomega $$
is obtained when $\bfomega(0) = {\bf 0}$.

\subsubsection{Semidirect Products}

Sometimes cost functions such as those in (\ref{twisted1}) arise naturally in the context of semi-direct products
rather than direct. The group of special Euclidean motions
$$ SE(3) = \IR^3 \rtimes SO(3) $$
is such semi-direct product. Its elements reside in the same underlying space as the pose change group, $g_i = (R_i, {\bf t}_i)$, but the
group law is different:
$$ g_1 g_2 = (R_1 R_2, R_1 {\bf t}_2 + {\bf t}_1)\,. $$
The standard way to describe $SE(3)$ as a matrix Lie group is to write its elements as
$$ g \,=\, \left(\begin{array}{cc}
R & {\bf t} \\
{\bf 0}^T & 1 \end{array}\right)\,\in\, \IR^{4\times4}. $$

The corresponding velocities can be written as
$$ \bfxi = (g^{-1} \dot{g})^{\vee} \,=\,
\left(\begin{array}{c}
\bfxi_i \\
\bfxi_2 \end{array}\right) \,=\,
\left(\begin{array}{c}
(R^T \dot{R})^{\vee} \\
R^T {\bf v} \end{array}\right) $$

For example, in modeling the minimal energy shapes of DNA molecules treated as elastic filaments, a common model
to describe the strain energy at the point $s$ along a backbone curve is of the form
$$ f(g,\bfxi) = \half \|\bfxi - \bfxi_0\|_{K} \,. $$
For special cases of $K$, the theory in (\ref{globalLie}) can be applied. The resulting geodesics will be
different than in the direct product case, with the translational parts taking helical paths. In principle, similar expressions can be used
when modeling large elastic deformations in sheets and solids comprised of elastic filaments in multiple directions. Then such a formulation
may be applied to robots powered by transparent elastic actuators, as in \cite{zhujian}.

For special cases of $A(g)$ it is also possible to capture a cost function such as (\ref{twisted1}) in the form
of $\half \|\bfxi\|_{K}$. In particular, writing
$$ \|\bfxi\|_{K} \,=\, \bfxi_1^T K_{11} \bfxi_1 \,+\, \|\bfxi_2 - A_0 \bfxi_1\|_W^2 $$
we recognize that $\bfxi_2 = R^T {\bf v}$, and when $W = c\II$, then
$$ \|\bfxi_1 - A_0 \bfxi_2\|_W^2 = c^2 \|\bfxi_2 - R A_{0}^{-1} \bfxi_1\|^2\,. $$
Then (\ref{twisted1}) can be written as a homogenous problem on $SE(3)$ rather than as a bootstrapped problem
from $SO(3)$ to $SO(3) \times \IR^3$. In other words, we can minimize the bootstrapped problem on
$SO(3) \times \IR^3$ as a homogeneous one on $SE(3)$ with $f(\bfxi) = \half \|\bfxi\|_{K}$ and
where
$$ K \,=\, \left(\begin{array}{cc}
K_{11} & c A_0 \\
c A_0^T & c\II \end{array}\right)\,. $$

%
%Note however, that
%(\ref{geopath}) for the semidirect product is \underline{not}  the same as for the direct product.
%
%This difference can be illustrated for $SE(3)$ vs. $SO(3) \times \IR^3$. The rotational parts of
%(\ref{geopath}) and (\ref{geopathdirectprod}), but in the former the translational part is a straight line,
%and in the latter it is a helix.

%\subsection{Globally Optimal Solutions on Cotangent Bundles}
%
%The affine group
%$$ Aff(3) = \IR^n \rtimes GL(n) $$
%is also a semi-direct product with
%group law
%$$ g_1 g_2 = (A_1 A_2, A_1 {\bf t}_2 + {\bf t}_1)\,. $$
%Like $SE(3)$, the standard way to describe $Aff(n)$ as a matrix Lie group is to write its elements as
%$$ g \,=\, \left(\begin{array}{cc}
%A & {\bf t} \\
%{\bf 0}^T & 1 \end{array}\right)\,\in\, \IR^{(n+1)\times (n+1)}. $$
%
%This group has importance in graphics, computer aided design, and computational geometry, among other fields.

\subsection{Application: Extraction of Salient Actions from a Video Sequence}

An $M\times N$ black-and-white image consists of $M\cdot N$ pixels, each with fixed location and pixel
intensity (grayscale) values. When there is color,
instead of a scalar, each pixel has a vector of RGB (red, green, blue) values.
A black-and-white video can be viewed as a curve in $\IR^{M\times N}$ parameterized by time,
and an individual image in the video is a point in this space. Similarly, a color video can be thought of as a curve
in $\IR^{M\times N \times 3}$.

Given a video sequence, a problem of importance is activity recognition. For example, if a humanoid robot is charged with taking
care of grandma, the robot should decode grandma's actions and gestures such as waving, walking, and summoning. As such, salient actions need to be extracted from video sequences and matched with stored representatives
of such action classes in a database. A first step in comparing two video sequences $X_1(\tau)$ and $X_2(\tau)$ for
$\tau \in [0,1]$ is to put the salient actions on a common timescale. Suppose that each video illustrates
a human hand waving from left to right and back. One person may wave faster at the beginning and slower at the end, and the other person
may wave slower at the beginning and faster at the end. This means that computing a quantity such as
$\int_{0}^{1} \|X_1(\tau) - X_2(\tau)\|^2 d\tau$ to measure the differences between the videos
will give a large value, and miss the fact that both videos describe essentially the same waving
simply because they are out of sync.
A way to reconcile the differences in the trajectories described by the videos is to assign a different timescale for one
relative to the other.
%For example, the duration of waving in one video may be quite different than in another, and the relative amount of dwell time at different parts of the wave might be different.
%To compare them they need to be put on a common timescale.
This is one of the goals of dynamic time warping. The literature on this topic is immense, and some well known articles include
\cite{DTW1,DTW2,DTW3}.

An alternative to dynamic time warping is to optimally reparameterize both of the videos as $X_i(\tau_i(t))$.
Since $dX_i/dt = X'_i(\tau) \dot{\tau}$, a measure of the amount of change in video $i$ at any instant is
$\|\dot{X}_i(\tau)\|^2  = \|X'_i(\tau)\|^2 \dot{\tau}^2$. An optimal timescale $\tau_i^*(t)$ for each video can be defined
as one that expands time when there is a lot of change in the video (i.e., large value of derivative), and compresses time
when nothing is happening (i.e., small value of derivative). In other words, we seek
$X_i(\tau_i^*(t))$ where
$\tau_i^*$ minimizes the functional
$$ J = \int_{0}^{1} \|X'_i(\tau)\|^2 \dot{\tau}^2 dt \,.$$
This is analogous to the curve reparametrization problem reviewed earlier in this paper, and
normalizes the timescale of both videos according to how much is changing in each.
This is the approach taken in \cite{gora1,gora2}.

The curvature of such an arclength parameterized curve in pixel space can be described using a higher-dimensional
version of the Frenet-Serret apparatus, and the analogs of curvature and torsion can be used to identify motion
features as explained in \cite{crane}.

Moreover, rather than considering only the raw images in the video sequence, it is common to
use image processing techniques to detect edges or extract other features that are tracked from frame to frame.
These features undergo motions within the image plane, which can be thought of as being orthogonal to the
temporal direction. Motions can be intrinsic to the the action being observed by the camera, or they
can be a result of nuisances, as would be the case if the camera is shaking when the video is taken, or one
person is waving from a stationary location and another is waving while riding a horse. When the goal is to isolate
actions, the contribution of these nuisance motions should be eliminated as much as possible before comparing videos.
It therefore makes sense to seek the optimal path $g_i(\tau) \in G$ (a group of nuisance motions)
such that $Y_i(\tau) = g_i(\tau) \cdot X_i(\tau)$ is minimally varying.
Minimizing $\int_{0}^{1} \|dY_i/d\tau\|^2 d\tau$ results in an Euler-Poincar\'{e} problem on $G$.
After nuisance motions are removed, the resulting $Y_i^*(\tau) = g_i^*(\tau) \cdot X_i(\tau)$ can be optimally
reparameterized as described above. The resulting motion-corrected and temporally reparameterized
videos $Y_i^*(\tau_i^*(t))$ can then be compared.

In this way a problem reduces to something analogous to the joint $\theta$-$s$ problem in curve framing
(with motion in the image plane being like $\theta(t)$, and compression or expansion of the temporal variable
being like $s(t)$). This joint problem was addressed in \cite{quick}.

\section{Conclusions}

A number of variational calculus problems with globally optimal solutions are reviewed and used
to construct new globally optimal solutions to variational problems on larger spaces. This
``bootstrapping'' approach is demonstrated in the context of simultaneous twist adjustment relative to the Frenet frames of a space curve to give the Bishop frames, together with the reparametrization of the framed curve starting with arclength to result in a framing that evolves more gradually. These ideas are then generalized to classes of coordinate-dependent variational problems and coordinate-free Lie-group problems with globally minimal solutions.
In many cases of interest one underlying space can be endowed with multiple different group
operations (e.g., direct product or semidirect product), and it is shown how different
globally minimal trajectories result from bootstrapping from subgroups to these larger product groups.

\section{Acknowledgements}

This work was performed under National Research Foundation, Singapore, under its Medium Sized Centre Programme - Centre for
Advanced Robotics Technology Innovation (CARTIN), subaward R-261-521-002-592, MOE Tier 2 grant REBOT A-8000424-00-00,
SMI Maritime R+D Grant A-8000081-02-00, CDE Board funds E-465-00-0009-001, and National University of Singapore Startup
grants A-0009059-02-00 and A-0009059-03-00. The author would like to thank Mr. Jikai Ye, Mr. Yuwei Wu, and Ms. Yi Xu for
checking equations.

\section{Data Availability Statement}

There is no data in this paper.

%\end{paracol}
%\reftitle{References}

\end{document}